\documentclass[11pt,a4paper,reqno]{article}
\usepackage[margin=30mm]{geometry}
\usepackage{amssymb,amsthm,enumerate,color}
\usepackage{amsmath,fancybox,amsfonts,ascmac,bm,ulem}
\RequirePackage{graphicx}
\usepackage[left]{lineno}
\newcommand*\patchAmsMathEnvironmentForLineno[1]{
  \expandafter\let\csname old#1\expandafter\endcsname\csname #1\endcsname
  \expandafter\let\csname oldend#1\expandafter\endcsname\csname end#1\endcsname
  \renewenvironment{#1}
     {\linenomath\csname old#1\endcsname}
     {\csname oldend#1\endcsname\endlinenomath}}
\newcommand*\patchBothAmsMathEnvironmentsForLineno[1]{
  \patchAmsMathEnvironmentForLineno{#1}
  \patchAmsMathEnvironmentForLineno{#1*}}
\AtBeginDocument{
\patchBothAmsMathEnvironmentsForLineno{equation}
\patchBothAmsMathEnvironmentsForLineno{align}
\patchBothAmsMathEnvironmentsForLineno{flalign}
\patchBothAmsMathEnvironmentsForLineno{alignat}
\patchBothAmsMathEnvironmentsForLineno{gather}
\patchBothAmsMathEnvironmentsForLineno{multline}
}

\usepackage{tikz}
\usetikzlibrary{intersections, calc}
\usetikzlibrary{decorations.markings}
\usetikzlibrary{arrows.meta}
\tikzset{
  arrows along my path/.style={
    postaction={
      decorate,
      decoration={
        markings,
        mark=between positions 0.03 and 1 step 10pt with {\arrow{Stealth[length=5pt]}},
   }}}}

\providecommand{\U}[1]{\protect\rule{.1in}{.1in}}
\theoremstyle{plain}
\newtheorem{thm}{Theorem}[section]

\newtheorem{cor}[thm]{Corollary}
\newtheorem{lem}[thm]{Lemma}

\newtheorem{prop}[thm]{Proposition}

\newtheorem{problem}[thm]{Problem}

\theoremstyle{definition}
\newtheorem{defn}[thm]{Definition}
\newtheorem{ex}[thm]{Example}
\theoremstyle{remark}
\newtheorem{rem}[thm]{Remark}
\newcommand{\C}{\mathbb{C}}
\newcommand{\R}{\mathbb{R}}

\newcommand{\N}{\mathbb{N}}

\newcommand{\calB}{\mathcal{B}}

\newcommand{\rmd}{\mathrm{d}}

\renewcommand{\Re}{{\sf Re}}
\renewcommand{\Im}{{\sf Im}}

\renewcommand{\leq}{\leqslant}

\renewcommand{\P}{\mathcal{P}(\mathbb{R})}

\begin{document}

\title{On freely quasi-infinitely divisible distributions}
\author{Ikkei Hotta\thanks{Department of Applied Science, Yamaguchi University 2-16-1 Tokiwadai, Ube 755-8611, Japan} \\
Wojciech M\l otkowski\thanks{Instytut Matematyczny, Uniwersytet Wroc{\l}awski, Plac Grunwaldzki 2/4, 50-384 Wroc{\l}aw, Poland} \\
 Noriyoshi Sakuma\thanks{Graduate School of Natural Sciences,
Nagoya City University,
Mizuho-ku, Nagoya, Aichi 467-8501, Japan} \\
  Yuki Ueda\thanks{Department of Mathematics, Hokkaido University of Education, 9 Hokumon-cho, Asahikawa 070-8621, Japan}}
\date{\today}
\maketitle
\begin{abstract}
Inspired by the notion of quasi-infinite divisibility (QID), we introduce and study
the class of freely quasi-infinitely divisible (FQID) distributions on $\mathbb{R}$,
i.e. distributions which admit the free L\'{e}vy-Khintchine-type representation with
signed L\'{e}vy measure. We prove several properties of the FQID class, some of them in contrast to
those of the QID class. For example, a FQID distribution may have negative Gaussian part,
and the total mass of its signed L\'{e}vy measure may be negative.
Finally, we extend the Bercovici-Pata bijection, providing a characteristic triplet,
with the L\'{e}vy measure having nonzero negative part, which is at the same time
classical and free characteristic triplet.
\end{abstract}

\tableofcontents

\section{Introduction}

In classical and free probability theories, infinitely divisible distributions are defined by the classical convolution operation and connected with L{\'e}vy processes, which are stochastic processes with independent increments and time-homogenous stationary distributions. 
More precisely, if $\mu$ is an infinitely divisible distribution on $\R$, then one can construct a L{\'e}vy process $\{X_t\}_{t\ge 0}$ such that a marginal law of $X_1$ coincides with $\mu$. 
Conversely, if a stochastic process $\{X_t\}_{t\ge0}$ is a L{\'e}vy process, then its marginal distribution is infinitely divisible. 
Because of this relation, infinite divisibility plays a crucial role in research on the marginal laws of L\'{e}vy processes. Readers may consult \cite{Sato} for the classical case, and \cite{BNT02,BNT06,Bia98} for the free case.

Infinitely divisible distributions are characterized in terms of analytic tools: the characteristic function in the classical case \cite{Sato} and the $R$-transform or the Voiculescu transform in the free case \cite{BNT06,BeVo1993}. These representations are called L\'{e}vy-Khintchine representations. 
One of the most beautiful discoveries in this area is that the L\'{e}vy-Khintchine representation in both classical and free probability can be parametrized by $a\ge0$, $\gamma\in\R$, and a L{\'e}vy measure $\nu$.
The triplet $(a,\nu,\gamma)$ is called characteristic triplet and free characteristic triplets, respectively.
They give a bijection between the classes of classical and free infinitely divisible distributions (see \cite[Theorem 1.2]{BP}).
It is called the Bercovici-Pata bijection now.
The importance of this bijection comes from the study of the limit theorem in free probability. 
Indeed, research on free infinitely divisible distributions has made rapid progress now that the rich structure of the class of free infinitely divisible distributions has been revealed. 

Going back to the classical probability case, many researchers have faced similar interesting examples in which the distributions are not infinitely divisible, but have a L\'{e}vy-Khintchine-type representation.
Namely, in the L{\'e}vy-Khintchine representation, the L\'{e}vy measure is a signed measure.
A precise definition is given in Section 2.
Such a distribution is called a quasi-infinitely divisible distribution.
Although the class of quasi-infinitely divisible distributions has been considered for a long time since being investigated in \cite{LO77}, few systematic approaches for this class have been reported. 
However, there has been some progress in the past decade. In particular, Lindner, Pan, and Sato in \cite{LPS} gathered many examples and studied the distributional properties of quasi-infinitely divisible distributions from a real analytic view.
They investigate quasi-infinitely divisible distributions from the L\'{e}vy-Khintchine representation and the characteristic triplet.
For example, we can not allow the negative Gaussian part $a$. If $a<0$, then there does not exists a corresponding probability measure with the characteristic triplet.

The purpose of this paper is to investigate the freely quasi-infinitely divisible (FQID) distributions. The class of FQID distributions is a natural extension of the class of freely infinitely divisible distributions to explore correspondence free and classical probability. 
Section 2 presents some preliminary results that are used in the remainder of the paper. 
In Section 3, we define the FQID distributions and identify several properties (convolution, atoms, convergence, and support) of FQID distributions. 
In Section 4, we present some examples of FQID distributions that are not freely infinitely divisible. Moreover, we discover different points about classical and free quasi-infinitely divisible distributions through several examples. 
In Section 5, we focus on a question raised by Bo{\.z}ejko of whether the Bercovici-Pata bijection can be extended to a larger class than one of infinitely divisible distributions. In order to answer his question, we find examples which come from the L{\'e}vy measure with the Cauchy distribution and symmetric distribution. 

\section{Preliminaries}

\subsection{Infinitely divisible distributions}

Let $\P$ be the set of all (Borel) probability measures on $\R$. For $\mu\in\P$, 
its characteristic function is defined by 
\begin{align*}
\widehat{\mu}(z):=\int_{\R}e^{izx}\mu(\rmd x), \qquad z\in \R.
\end{align*} 
A probability measure $\mu\in\P$ is said to be {\it (classical) infinitely divisible} (ID) if, for any $n\in \N$, there exists some $\mu_n\in\P$ such that $\mu=\mu_n^{\ast n}$. The class of ID distributions, denoted by ${\rm ID}(\ast)$, comes from limit theorems and plays a crucial role in constructing stochastic processes in connection to the study of L{\'e}vy processes and its applications in mathematical statistics, finance, and physics.

If $\mu$ is infinitely divisible, then its characteristic function admits the following representation (namely, the {\it L\'{e}vy-Khintchine representation}):
\begin{align}\label{cLK}
\widehat{\mu}(z)=\exp\left(i\gamma z-\frac{a}{2}z^2+\int_\R (e^{i zx}-1-izx \mathbf{1}_{[-1,1]}(x))\nu(dx)\right), \qquad z\in \R,
\end{align}
where $\gamma\in\R$, $a\ge 0$, and $\nu$ is a L\'{e}vy measure\footnote{A (Borel-) measure $\nu$
  on $\R$ is called a L\'evy measure if $\nu(\{0\})=0$ and
  $\int_{\R}(1\land x^2)\,\nu(\rmd x)<\infty$.} on $\R$. 
  Conversely, if the characteristic function of $\mu\in\P$ is given by the L\'{e}vy-Khintchine representation, then $\mu$ is ID. 
Each $\mu\in\P$ determines a unique triplet $(a,\nu,\gamma)$, called a {\it characteristic triplet for $\mu$}; see \cite{Sato} for further information.

The characteristic function of an ID distribution has another representation.
Denote the class of all Borel sets in $\R$ as $\calB$.
Let $\mu$ be an ID distribution with a characteristic triplet $(a,\nu,\gamma)$. If we set $b:=\gamma$ and define a finite measure $\tau$ by
  \begin{align*}
  \tau(B):=a\delta_0(B)+\int_B(1\land x^2)\nu(dx), \qquad B\in\calB,
  \end{align*}
  then 
  \begin{align}\label{acLK}
  \widehat{\mu}(z)=\exp\left( ib z+\int_\R g_c(x,z)\tau(dx)\right), \qquad z\in\R,
  \end{align}
where the function $g_c$ is defined by
\begin{align*}
g_c(x,z):=\begin{cases}
(e^{izx}-1-izx\mathbf{1}_{[-1,1]}(x))/(1\land x^2), & x\neq 0,\\
-z^2/2, & x=0, 
\end{cases}
\qquad z\in \R.
\end{align*}
Observe that the kernel function $\R\rightarrow \R: x\mapsto g_c(x, z)$ in \eqref{acLK} is bounded for each $z\in \R$ and continuous at $0$. It is easy to see that $\mu\in\P$ determines a unique pair $(b,\tau)$, called the {\it characteristic pair for $\mu$}. Conversely, if a characteristic function of $\mu$ is given by \eqref{acLK} for some characteristic pair $(b,\tau)$, then it also has the L\'{e}vy-Khintchine representation \eqref{cLK} for some characteristic triplet $(a,\nu,\gamma)$.

\subsection{Quasi-infinitely divisible distributions}
Here, we summarize the work of Lindner, Pan, and Sato \cite{LPS} for quasi-infinitely divisible distributions. 

A function $\nu:\calB\to[-\infty,\infty]$ is called a {\it signed measure} on $\R$ if $\nu(\emptyset) =0$ and $\nu(\bigcup_{n=1}^{\infty} A_{n}) = \sum_{n=1}^{\infty}\nu(A_{n})$ for any pairwise disjoint sets $\{A_n\}_n$ in $\calB$. 
We define the {\it total variation} of a signed measure $\nu$ by
\begin{align*}
|\nu|(A) := \sup_{\{A_{j}\}: \text{partition of }A} \sum_{j} |\nu(A_{j})|.
\end{align*}
It is known that $|\nu|$ is a measure on $\R$.
A signed measure $\nu$ is said to be {\it finite} if $|\nu|$ is a finite measure. 
By the Hahn-Jordan decomposition theorem \cite[p.127]{Rudin}, for a finite signed measure $\nu$, there exist disjoint Borel sets $C^{+}, C^{-}$ in $\R$ and finite measures $\nu^{+},\nu^{-}$ on $\calB$ satisfying $\nu = \nu^{+}-\nu^{-}$ so that $\nu^{+}(\R\backslash C^{+}) = \nu^{-}(\R\backslash C^{-})=0$.
We call $\nu^{+}$ and $\nu^{-}$ the {\it positive part} and the {\it negative part} of the signed measure $\nu$, respectively. 
The measures $\nu^+$ and $\nu^-$ are uniquely determined by $\nu$.

We now define quasi-infinitely divisible distributions on $\R$.
\begin{defn}
A probability measure $\mu\in\P$ is called a {\it quasi-infinitely divisible (QID) distribution} if its characteristic function $\widehat{\mu}$ admits the representation \eqref{acLK} for some $b \in \R$ and a finite signed measure $\tau$ on $\R$. 
\end{defn}
A pair $(b,\tau)$ is uniquely determined by a QID distribution $\mu$ and is called a {\it characteristic pair} for a QID distribution $\mu$. 
$\mu$ is QID if and only if there exist $\mu_1,\mu_2\in I(\ast)$ such that $\mu\ast \mu_1=\mu_2$ by applying the Hahn-Jordan decomposition to a finite signed measure $\tau$ in \eqref{acLK}.

We can rephrase this in terms of the characteristic triplets. 
We seek a characteristic triplet for the QID distribution $\mu$ in terms of the L\'{e}vy-Khintchine representation in the sense of \eqref{cLK}. 
Here, we have to be a little careful.
First, we assume that there exist $\mu_1,\mu_2\in I(\ast)$ such that $\mu\ast \mu_1=\mu_2$ and $\nu_{1}$ and $\nu_{2}$ are the Levy measures of $\mu_{1}$ and $\mu_{2}$, respectively.
If $\nu_1$ and $\nu_2$ are infinite, the difference $\nu_1-\nu_2$ is no longer a signed measure. 
To define a signed L\'{e}vy measure (namely, a quasi-L\'{e}vy measure), we define the following families: $\calB_{r}:=\{ B \in\calB ; B \cap (-r,r)=\emptyset \}$ for $r>0$ and $\calB_{0} := \bigcup_{r>0}\calB_{r}$.
Note that $\calB_{0}$ is not a $\sigma$-algebra, so we must take care to define signed measures on $\R\backslash\{0\}$. 

A function $\nu:\calB_{0}\to[-\infty,\infty]$ is called a {\it quasi-L{\'e}vy type measure} if 
\begin{enumerate}[{\rm (i)}]
\item
$\nu|_{\calB_{r}}$ is a finite signed measure for each $r>0$;
\item
$\int_{\R} (1\wedge x^{2})|\nu|(dx)<\infty$.
\end{enumerate}

Condition (i) ensures that $|\nu|$, $\nu^{+}$, and $\nu^{-}$ are unique and well-defined measures on $(\R,\calB)$ satisfying
\begin{align*}
|\nu|(\{0\}) = \nu^{+}(\{0\}) = \nu^{-}(\{0\})=0,
\end{align*}
which, in particular, justifies the integral in condition (ii). Moreover,
\begin{align*}
&|\nu|(A) =\Big|\nu|_{\calB_{r}}\Big|(A) = \Big|\nu|_{\calB_{s}}\Big| (A)\quad\text{ for \, $A\in\calB_{r}$\, and \,$s<r$,}\\
&\nu^{+}(A) =\nu^{+}|_{\calB_{r}}(A) = \nu^{+}|_{\calB_{s}}(A) \text{ \quad for\, $A\in\calB_{r}$ \,and\, $s<r$},\\
&\nu^{-}(A) =\nu^{-}|_{\calB_{r}}(A) = \nu^{-}|_{\calB_{s}}(A) \text{\quad for\, $A\in\calB_{r}$ \,and\, $s<r$.}
\end{align*}

For a QID distribution $\mu$ with a characteristic pair $(b,\tau)$, we define a function $\nu:\mathcal{B}_0\rightarrow \R$ by
\begin{align*}
\nu(B):=\int_B (1\land x^2)^{-1} \tau(dx), \qquad B\in \calB_{0}.
\end{align*}
The measure $\nu$ is then a quasi-L\'{e}vy-type measure. Moreover, setting $\gamma:=b$ and $a:=\tau(\{0\})$, the characteristic function of the QID distribution $\mu$ admits the following representation: 
\begin{equation}
\label{q-LK-eq}
\widehat{\mu}(z)=\exp\left(-\frac{a}{2}z^2 + i\gamma z+\int_\R (e^{i zx}-1-izx \mathbf{1}_{[-1,1]}(x))\nu(dx)\right), \qquad z\in\R.
\end{equation}
Conversely, if $\mu\in\P$ has the representation given in \eqref{q-LK-eq} with $a, \gamma \in \R$ and a quasi-L\'{e}vy-type measure $\nu$ on $\R$, then we obtain the representation given by \eqref{acLK} for some characteristic pair $(b,\tau)$, and therefore $\mu$ is QID. The triplet $(a,\nu,\gamma)$ is called a {\it characteristic triplet} of the QID distribution $\mu$. 

In particular, a quasi-L\'{e}vy-type measure $\nu$ is called a {\it quasi-L\'{e}vy measure} for a QID distribution $\mu$ if the characteristic function of $\mu$ has the representation \eqref{q-LK-eq} for some characteristic triplet $(a,\nu,\gamma)$. For any QID distribution, $a\ge 0$. 
Not every quasi-L\'{e}vy-type measure is a quasi-L\'{e}vy measure (see, e.g., \cite[Example 2.9]{LPS}).

We say that a function $f:\R\rightarrow\R$ is {\it integrable with respect to a quasi-L\'{e}vy-type measure $\nu$} if it is integrable with respect to $|\nu|$ (and therefore with respect to $\nu^+$ and $\nu^-$ as well).
We then define
\begin{align*}
\int_B f(x)\nu(dx):=\int_B f(x)\nu^+(dx)-\int_Bf(x)\nu^-(dx),\qquad B\in \calB.
\end{align*}
Note that, for each $z\in\R$, the function $x\mapsto e^{izx}-1-iz \mathbf{1}_{[-1,1]}(x)$ is integrable with respect to $\nu$.

\subsection{Free probability theory}

In non-commutative probability theory, self-adjoint operators are interpreted as (real-valued) random variables. A remarkable feature is that various notions of independence exist for those random variables. In particular, free independence has found applications in operator algebras and random matrices, and has been intensively studied (see \cite{NiSpBook,MS17}).

\subsubsection{Free additive convolution}
Let $X$ and $Y$ be freely independent random variables affiliated with a von Neumann algebra $\mathcal{M}$ equipped with a normal faithful tracial state $\tau$. Note that, by the definition of the affiliated operator, $E_X(B)$ and $E_Y(B)$ belong to $\mathcal{M}$ for any Borel set $B$ in $\R$, where $E_X$ is the spectral projection of $X$. Denote a probability distribution of $X$ with respect to $\tau$ as $\mu_X$, i.e., $\mu_X((-\infty,\cdot])=\tau(E_X((-\infty,\cdot])))$. 
A probability distribution of $X+Y$ with respect to $\tau$ is denoted by $\mu_X\boxplus \mu_Y$ with the operation $\boxplus$, called {\it free additive convolution}. This was first introduced in \cite{V86}, where $\mu$ and $\nu$ have compact supports. The concept was later extended to probability measures with finite moments (see \cite{Maa}). Finally, it was generalized by Bercovici and Voiculescu \cite{BeVo1993} to the case of probability measures $\mu$ and $\nu$ on $\R$. Denote the $n$-fold free additive convolution of $\mu\in\P$ as $\mu^{\boxplus n}$. 
 
Free additive convolution is characterized by the Voiculescu transform (see \cite{BeVo1993}).  A {\it Cauchy-Stieltjes transform} $G_\mu$ of $\mu\in\P$ is defined by 
\begin{equation*}
G_\mu(z)=\int_{\R}\frac{1}{z-x}\,\mu(\rmd x), \qquad z\in\C^+,
\end{equation*}
where $\C^+$ (resp.\ $\C^-$) denotes the set of complex numbers with
strictly positive (resp.\ strictly negative) imaginary parts.
Note that $\Im(G_\mu(z))<0$ for any $z \in \C^+$. 
A reciprocal Cauchy transform $F_{\mu}:=1/G_{\mu}$ is called an \textit{F-transform}. 
It is easy to see that $F_\mu$ is a Pick function, namely, an analytic function from $\C^+$ to $\C^+$, and $F_\mu(iy)/iy\rightarrow 1$ as $y\rightarrow\infty$.
Conversely, a Pick function is characterized by an F-transform of some probability measure on $\R$:
\begin{prop}\label{exist}(see \cite[Proposition 5.2]{BeVo1993})
If an analytic function $F:\C^{+}\to \C^{+}$ satisfies
$$
\lim_{y\to \infty} \frac{F(iy)}{iy} = 1,
$$
then there exists some $\mu\in\P$ such that 
$F=F_{\mu}$.
\end{prop}

In classical probability theory, additive convolution $\ast$ is characterized by the cumulant transform: for $\mu,\nu\in \P$ we have that $\log\widehat{\mu\ast\nu}(z)=\log\hat{\mu}(z)+\log\hat{\nu}(z)$ for all $z\in \R$ (see e.g., \cite{Sato}). In free probability theory, there is a similar characterization of free additive convolution using the free analog of the cumulant transform. For any $\mu\in\P$ and $\lambda>0$,  there exist 
positive numbers $\alpha,\beta$, and $M$ such that $F_{\mu}$ is univalent  
on the set $\Gamma_{\alpha,\beta}:=\{z \in \C^{+} \,|\, \Im(z) >\beta,
|\Re(z)|<\alpha \Im(z)\}$ and 
$F_{\mu}(\Gamma_{\alpha,\beta})\supset\Gamma_{\lambda,M}$. Thus, the compositional right inverse $F^{-1}_{\mu}$ is defined on
$\Gamma_{\lambda,M}$. The {\it Voiculescu transform} $\varphi_\mu$ is defined by 
\begin{align*}
 \varphi_\mu(z):=F_\mu^{-1}(z)-z, \qquad z\in \Gamma_{\lambda,M}.
\end{align*}
It has been shown \cite{BeVo1993} that $\varphi_{\mu\boxplus \nu}(z)=\varphi_\mu(z)+\varphi_\nu(z)$ 
on the intersection of the domains where $\varphi_\mu$ and $\varphi_\nu$ are defined.

\subsubsection{Free infinitely divisible distributions}

Since the work of Bercovici and Pata \cite{BP}, who found a bijection between the classical and freely infinitely divisible distributions, the class of infinitely divisible distributions in free probability has been intensively studied. 
A probability measure $\mu\in\P$ is said to be {\it freely infinitely divisible} (FID) if, for any $n\in \N$, there exists some $\mu_n\in\P$ such that $\mu=\mu_n^{\boxplus n}$. The class of FID distributions is written as ${\rm ID}(\boxplus)$, introduced in \cite{BeVo1993}. Recently, many FID distributions have been discovered: a normal distribution \cite{BBLS11}, some classical stable laws \cite{HSW}, some gamma distributions that include the chi-square distributions \cite{Has14}, a classical Meixner distribution \cite{BH13}, and the Fuss-Catalan distribution \cite{MSU20}.

FID distributions can be characterized by those Voiculescu transforms via complex analysis. 
Bercovici and Voiculescu \cite{BeVo1993} gave the following
important result on FID distributions:

\begin{prop}\label{prop:Voiculescu transform}
A probability measure $\mu\in\P$ is in ${\rm ID}(\boxplus)$ if and only if
the Voiculescu transform $\varphi_{\mu}$ has an analytic extension
defined on $\C^{+}$ with values in $\C^{-}\cup \R$.
\end{prop}

Namely, if $\mu$ is FID, then the function $z\mapsto -\varphi_\mu(z)$ has an analytic extension to the Pick function (denoted by the same symbol $\varphi_\mu$). Then, for $\mu\in {\rm ID}(\boxplus)$, its Voiculescu transform $\varphi_\mu$ has the following representation (called a {\it Pick-Nevanlinna representation}):
\begin{align*}
\varphi_\mu(z)=b+\int_\R \frac{1+xz}{z-x}\tau(dx), \qquad z\in \C^+,
\end{align*}
for some $b\in\R$ and a finite measure $\tau$ on $\R$. The pair $(b,\tau)$ is uniquely determined by the FID distribution $\mu$, and is called the {\it free characteristic pair} (or {\it free generating pair}, see \cite[Definition 2.9]{BNT02}) for the FID distribution $\mu$.

The FID distribution is also characterized by admitting the L\'{e}vy-Khintchine representation in terms of the $R$-transform. The {\it R-transform} (or {\it free cumulant transform}) $R_\mu$ of $\mu\in\P$ is defined by
\begin{align*}
 R_\mu(z)=z\varphi_\mu\left(\frac{1}{z}\right), \quad\text{for all $z\in \C^-$ such that
  $1/z \in \Gamma_{\lambda,M}$}.
\end{align*}
Note that $1/z\in \Gamma_{\lambda, M}$ is equivalent to $z\in \Delta_{\lambda,1/M}$, where
\begin{align}\label{Delta}
\Delta_{\alpha,\beta}:=\{z\in \C^-: |\Re(z)| <-\alpha \Im(z), |z|<\beta\}, \qquad \alpha,\beta>0.
\end{align}

The free version of the L\'evy-Khintchine representation amounts
to the statement that $\mu\in {\rm ID}(\boxplus)$ if and only if there exist 
$\eta\in\R$, $a\ge 0$, and a L{\'e}vy measure
$\nu$ such that  
\begin{equation}
\label{free-LK-rep}
 R_{\mu}(z) = a z^{2}+ \gamma z+ 
\int_{\R}\left(\frac{1}{1- z x}-1-z x \mathbf{1}_{[-1,1]}(x)\right)\nu(\rmd x), \qquad z\in \C^-.
\end{equation}
The triplet $(a,\nu, \gamma)$, referred to as the {\it free characteristic triplet} for $\mu$, is uniquely determined by $\mu \in {\rm ID}(\boxplus)$. 
The measure $\nu$ is said to be the {\it free L\'evy measure} for $\mu$. 
Furthermore, the representation in \eqref{free-LK-rep} is called the {\it free L\'{e}vy-Khintchine representation} of an FID distribution $\mu$. The relation between the free characteristic triplet $(a,\nu,\gamma)$ and free characteristic pair $(b,\tau)$ for the FID distribution $\mu$ is as follows (see \cite{BNT06}):
\begin{equation}\label{Triplet}
\begin{split}
&a=\tau(\{0\}),\\
&\nu({\rm d}x)=\frac{1+x^2}{x^2}\cdot \mathbf{1}_{{\mathbb R}\setminus\{0\}}(x) \
\tau({\rm d}x),\\
&\gamma=b+\int_{{\mathbb R}}x\Big(\mathbf{1}_{[-1,1]}(x)-\frac{1}{1+x^2}\Big) \
\nu({\rm d}x).
\end{split}
\end{equation}

Let $\mu\in\P$ be given. 
Then, a discrete semigroup $\{\mu^{\boxplus n}\}_{n\in\N}$ can be embedded in a continuous family $\{\mu_t\}_{t\ge 1}$ of probability measures on $\R$, so that $\mu_1=\mu$ and $\mu_t\boxplus \mu_s=\mu_{t+s}$ for $t,s\ge 1$. 
Initially, Bercovici and Voiculescu \cite{BV} showed the existence of $\{\mu_{t}\}$ for sufficiently large $t$ and for $\mu \in \P$ having a compact support. 
Later, Nica and Speicher \cite{NS96} improved their result to $t\ge 1$, but $\mu$ still required a compact support.
Finally, Belinschi and Bercovici \cite{BB04} extended this further to a general probability measure.

This property can be expressed in the following way:
for any $\mu\in\P$ and $t\ge 1$, there exists $\mu_t\in\P$ such that $R_{\mu_t}(z)=tR_\mu(z)$ on the intersection of the domains of $R_{\mu_t}$ and $R_\mu$. 
We denote the above probability measure $\mu_t$ as $\mu^{\boxplus t}$. Furthermore, Bercovici and Voiculescu \cite{BeVo1993} showed that a partial semigroup $\{\mu^{\boxplus t}\}_{t\ge 1}$ can be embedded in the continuous semigroup $\{\mu^{\boxplus t}\}_{t\ge0}$, so that $\mu^{\boxplus 0}=\delta_0$, $\mu^{\boxplus 1}=\mu$, and $\mu^{\boxplus t}\boxplus \mu^{\boxplus s}=\mu^{\boxplus t+s}$ for $t,s\ge 0$ when $\mu\in I(\boxplus)$. This satisfies $R_{\mu^{\boxplus t}}(z)=tR_\mu(z)$ for all $z\in \C^-$. 

\subsubsection{The Bercovici-Pata bijection}

Let $\Lambda: {\rm ID}(\ast)\to {\rm ID}(\boxplus)$ be the Bercovici-Pata bijection (see  \cite[Theorem 1.2]{BP} and \cite{BNT06}). This bijection satisfies
\begin{align*}
\Lambda(\mu\ast \nu)=\Lambda(\mu)\boxplus \Lambda(\nu)
\end{align*}
for all $\mu,\nu\in {\rm ID}(\ast)$, and is continuous with respect to the weak convergence. Moreover, $\Lambda(\mu)\in {\rm ID}(\boxplus)$ has a free characteristic triplet $(a,\nu,\gamma)$ for $\mu\in {\rm ID}(\ast)$ with a characteristic triplet $(a,\nu,\gamma)$. For example, we obtain
\begin{enumerate}[(1)]
\item $\Lambda(\mathbf{N}(m,\sigma^2))=\mathbf{S}(m,\sigma^2)$ for $m\in\R$ and $\sigma>0$;
\item $\Lambda(\mathbf{Po}(\lambda))=\mathbf{MP}(\lambda)$ for $\lambda>0$;
\item $\Lambda(\mathbf{C}_a)=\mathbf{C}_a$ for $a>0$,
\end{enumerate}
where $\mathbf{N}(m,\sigma^2)$ is the normal distribution with mean $m\in \R$ and variance $\sigma^2$, $\mathbf{S}(m,\sigma^2)$ is the semicircle law with mean $m\in\R$ and variance $\sigma^2$, $\mathbf{Po}(\lambda)$ is the Poisson distribution with a parameter $\lambda>0$, $\mathbf{MP}(\lambda)$ is the Marchenko-Pastur law with a paramater $\lambda>0$ and $\mathbf{C}_a$ is the symmetric Cauchy distribution with parameter $a>0$. 

\section{Freely quasi-infinitely divisible distributions}

\subsection{Definition and convolution properties}

In this section, we introduce FQID distributions.

By Proposition \ref{prop:Voiculescu transform}, $\mu\in\P$ is FID if and only if the Voiculescu transform $\varphi_\mu$ has an analytic extension to $\C^+$ with values in $\C^-\cup \R$.
FQID distributions are defined by employing a finite  ``signed''  measure $\tau$ in the  Pick-Nevanlinna representation of $\varphi_\mu$.

\begin{defn}
A probability measure $\mu\in\P$ is said to be {\it freely quasi-infinitely divisible} if
\begin{align*}
\varphi_\mu(z)=b+\int_{\R} \frac{1+xz}{z-x}\tau(dx), \qquad z\in \C^+,
\end{align*} 
for some $b\in\R$ and a finite signed measure $\tau$ on $\R$. We call $(b,\tau)$ a {\it free characteristic pair} of the FQID distribution $\mu$.
\end{defn}

One immediate consequence of the definition is the following property.
\begin{prop}\label{prop:nonFQID}
If $\mu$ is FQID, then the Voiculescu transform $\varphi_\mu$ has an analytic continuation to $\C^+$. 
\end{prop}

Below, we state and prove some basic properties of FQID distributions.

\begin{prop}\label{uniquegenerating}
For an FQID distribution, the free characteristic pair is uniquely determined.
\end{prop}
\begin{proof}
Let $\mu$ be an FQID distribution on $\mathbb{R}$, and let $(b_{1}, \tau_{1})$ and 
$(b_{2}, \tau_{2})$ be free characteristic pairs of $\mu$. 
Let $(\tau_1^+,\tau_1^-)$ and $(\tau_2^+,\tau_2^-)$ be pairs of finite nonnegative measures given by the Hahn-Jordan decomposition of $\tau_1$ and $\tau_2$, respectively. 
Then, we have that
\begin{align*}
b_1+&\int_\mathbb{R}\frac{1+xz}{z-x} \tau_1^+(dx)-\int_\mathbb{R}\frac{1+xz}{z-x}\tau_1^-(dx)\\
&=\varphi_\mu(z)=b_2+\int_\mathbb{R}\frac{1+xz}{z-x} \tau_{2}^+(dt)-\int_\mathbb{R}\frac{1+xz}{z-x} \tau_2^-(dt), \qquad z\in \C^+.
\end{align*}
Therefore, we have
\begin{align*}
(b_1-b_2)+\int_\mathbb{R}\frac{1+xz}{z-x}(\tau_1^++\tau_2^-)(dt)=\int_\mathbb{R}\frac{1+xz}{z-x} (\tau_1^{-}+\tau_2^{+})(dt), \qquad z\in \C^+.
\end{align*}
Both sides correspond to the Voiculescu transform of some FID probability measure $\nu$. Hence, the pairs $(b_{1}-b_{2}, \tau_{1}^{+}+\tau_{2}^{-})$ and $(0, \tau_{1}^{-}+\tau_{2}^{+})$ are free characteristic pairs of the FID distribution $\nu$. 
By the uniqueness of free characteristic pairs of FID distributions (see \cite{BeVo1993}), we have that $b_1-b_2=0$ and $\tau_1^++\tau_2^-=\tau_1^-+\tau_2^+$. Therefore, $b_{1}=b_{2}$ and $\tau_{1}=\tau_{2}$.
\end{proof}

Let $cB=\{ cx \,|\, x \in B\}$.
If $\rho$ is a Borel measure on $\R$ and $c$ is a nonzero real constant, then the dilation of $\rho$ by $c$ is the measure $\mathbf{D}_{c}(\rho)$ given by
\begin{align*}
\mathbf{D}_{c}(\rho)(B) = \rho(c^{-1}B),
\end{align*}
for any Borel set $B$. 

\begin{prop}\label{basic} We have the following properties.
\begin{enumerate}[\rm (i)]
\item $\mu$ is FQID if and only if there exist $\mu_1,\mu_2\in {\rm ID}(\boxplus)$ such that $\mu\boxplus \mu_1=\mu_2$.
\item If $\mu$ is FQID and $\mu^{\boxplus t}$ exists for some $t>0$, then $\mu^{\boxplus t}$ is also FQID.
\item If $\mu,\nu$ are FQID, $c\ne 0$ and $a\in\R$, then so are $\mu\boxplus \nu$, $\mathbf{D}_c(\mu)$ and $\mu\boxplus \delta_a$.
\item If $\mu$ is FID and $\mu\boxplus \nu$ is FQID, then $\nu$ is also FQID.
\item If $\mu$ is FQID and $\nu$ is not FQID, then $\mu\boxplus \nu$ is not FQID.
\end{enumerate}
\end{prop}

\begin{proof}
(i) is obvious from the Hahn-Jordan decomposition theorem. (ii) is a consequence of (i).

We prove (iii). 
By (i), there exist $\mu_1,\mu_2,\nu_1,\nu_2\in {\rm ID}(\boxplus)$ such that $\mu_1\boxplus \mu=\mu_2$ and $\nu_1\boxplus \nu=\nu_2$. Therefore,
\begin{align*}
(\mu_1\boxplus \nu_1)\boxplus (\mu\boxplus\nu)=(\mu_1\boxplus\mu)\boxplus(\nu_1\boxplus\nu)=\mu_2\boxplus \nu_2.
\end{align*}
Since $\mu_1\boxplus\nu_1, \mu_2\boxplus\nu_2\in {\rm ID}(\boxplus)$, the measure $\mu\boxplus\nu$ is FQID.
Using (i) again, we have that $\mu_1\boxplus\mu=\mu_2$. Hence,
\begin{align*}
\mathbf{D}_c(\mu_1)\boxplus \mathbf{D}_c(\mu)=\mathbf{D}_c(\mu_1\boxplus\mu)=\mathbf{D}_c(\mu_2).
\end{align*}
As $\mathbf{D}_c(\mu_1),\mathbf{D}_c(\mu_2)\in {\rm ID}(\boxplus)$, the measure $\mathbf{D}_c(\mu)$ is FQID.
Furthermore, it follows from (i) that $\mu\boxplus \mu_1=\mu_2$, and so
\begin{align*}
(\mu\boxplus \delta_a )\boxplus \mu_1= \delta_a\boxplus(\mu\boxplus\mu_1)=\delta_a\boxplus\mu_2.
\end{align*}
Since $\delta_a\boxplus \mu_1, \delta_a\boxplus\mu_2\in {\rm ID}(\boxplus)$, the measure $\mu\boxplus \delta_a$ is also FQID.

Next, we show (iv). As $\mu\boxplus\nu$ is FQID, there exist $\mu_1,\mu_2\in {\rm ID}(\boxplus)$ such that 
\begin{align*}
(\mu_1\boxplus \mu)\boxplus \nu=\mu_1\boxplus (\mu\boxplus \nu)=\mu_2.
\end{align*}
As $\mu_1\boxplus\mu$ is FID, we conclude that $\nu$ is FQID.

Finally, we show (v). Assume that $\mu\boxplus \nu$ is FQID. We then obtain $\sigma_1,\sigma_2\in {\rm ID}(\boxplus)$ such that $(\mu\boxplus \nu)\boxplus \sigma_1=\sigma_2$ by (i). Moreover, there exist $\mu_1,\mu_2\in {\rm ID}(\boxplus)$ such that $\mu\boxplus \mu_1=\mu_2$, because $\mu$ is also FQID. Then,
\begin{align*}
\sigma_2\boxplus \mu_1&=\sigma_1\boxplus (\mu\boxplus \nu)\boxplus \mu_1\\
&=\sigma_1\boxplus (\mu\boxplus \mu_1)\boxplus\nu=(\sigma_1\boxplus\mu_2)\boxplus \nu.
\end{align*}
Since $\sigma_2\boxplus \mu_1, \sigma_1\boxplus\mu_2 \in {\rm ID}(\boxplus)$, the measure $\nu$ is FQID, which is a contradiction.
\end{proof}

\begin{defn}
For $\mu_1, \mu_2, \mu\in \P$, we denote by $\mu=\mu_2\boxminus \mu_1$ the case in which 
$$
\varphi_\mu(z)=\varphi_{\mu_2}(z)-\varphi_{\mu_1}(z)
$$ 
holds for $z$ in some domain of $\C^+$. The operation $\boxminus$ is called {\it free deconvolution} (see \cite{ATV} for more information). 
\end{defn}

Using the operator $\boxminus$, we obtain a similar property to that of Proposition \ref{basic}(i), whereby \textit{$\mu$ is FQID if and only if $\mu = \mu_2\boxminus \mu_1$ for some FID distributions $\mu_1,\mu_2$}.

\subsection{Atoms of FQID distributions}

\begin{thm}\label{atom}
An FQID distribution has at most one atom.
\end{thm}

Theorem \ref{atom} can be shown by following the same argument as in \cite[Proposition 5.12(iii)]{BeVo1993}, because the Voiculescu transform of an FQID distribution has an analytic continuation to $\C^+$. Nevertheless, we include a proof for readers' convenience.

First, we need the following lemma.

\begin{lem}\label{lem:FQIDatom}
Let $\mu$ be an FQID distribution. Then,
\begin{enumerate}[\rm (i)]
\item We have that $F_\mu(z)+\varphi_\mu(F_\mu(z))=z$ for all $z\in\C^+$.
\item Let $\Omega_\mu$ be the component of $\{z\in\C^+: \Im(z+\varphi_\mu(z))>0\}$ that contains $iy$ for some large number $y>0$. Then, $F_\mu(\C^+)=\Omega_\mu$.
\end{enumerate}
\end{lem}

\begin{proof}
(i) Note that $\varphi_\mu(z)=F_\mu^{-1}(z)-z$ for all $z\in\Gamma_{\lambda,M}$ for some $\lambda,M>0$. Therefore,
\begin{align*}
\varphi_\mu(F_\mu(z))=z-F_\mu(z),\qquad z\in F_\mu^{-1}(\Gamma_{\lambda,M})\subset \C^+.
\end{align*}
Since $\varphi_\mu$ has an analytic continuation to $\C^+$ (denoted by the same symbol $\varphi_\mu$), we obtain
\begin{align*}
F_\mu(z)+\varphi_\mu(F_\mu(z))=z, \qquad z\in\C^+
\end{align*}
by the identity theorem. 

(ii) Take $w\in F_\mu(\C^+)$ and $z\in \C^+$ such that $w=F_\mu(z)$. Then,
\begin{align*}
w+\varphi_\mu(w)=F_\mu(z)+\varphi_\mu(F_\mu(z))=z\in\C^+,
\end{align*}
and therefore $\Im (w+\varphi_\mu(w))>0$. Thus, $F_\mu(\C^+)\subset \Omega_\mu$. 

Recall that $F_\mu(z+\varphi_\mu(z))=z$ for all $z\in F_\mu(\C^+)$. The identity extends by analytic continuation to $\Omega_\mu$. Finally, for all $z\in\Omega_\mu$, we have that $z+\varphi_\mu(z)\in\C^+$, and therefore $z=F_\mu(z+\varphi_\mu(z))\in F_\mu(\C^+)$. This means that $\Omega_\mu\subset F_\mu(\C^+)$.
\end{proof}

\begin{proof}[Proof of Theorem \ref{atom}]
Assume that $\mu$ has an atom $a\in\R$ with the mass $\beta=\mu(\{a\})>0$. We obtain
\begin{align*}
|G_\mu(a+iy)|\ge |\Im G_\mu(a+iy)|=\left|\int_\R \frac{-y}{(a-t)^2+y^2}d\mu(t) \right|\ge \frac{\beta}{y}.
\end{align*}
We then have that
\begin{align*}
|\Re F_\mu(a+iy)|\le |F_\mu(a+iy)|\le \frac{y}{\beta}\le \frac{\Im F_\mu(a+iy)}{\beta},
\end{align*}
where the last inequality holds by \cite[Corollary 5.3]{BeVo1993}. Therefore, $F_\mu(a+iy)\in \Gamma_{1/\beta}$ for all $y>0$. Hence, the curve $C:=F_\mu(a+i (0,\infty))$ approaches zero nontangentially. Moreover, if we define $u_\mu(z):=z+\varphi_\mu(z)$, the function $u$ maps $C$ to $\C^+$, because $C\subset F_\mu(\C^+)=\Omega_\mu$ holds by Lemma \ref{lem:FQIDatom}(ii). Finally, Lemma \ref{lem:FQIDatom}(i) implies that
\begin{align*}
\lim_{z\rightarrow 0, z\in C}u_\mu(z)=\lim_{y\rightarrow 0} (F_\mu(a+iy)+\varphi(F_\mu(a+iy)))=\lim_{y\rightarrow 0} (a+iy)=a.
\end{align*}
If the number of atoms of $\mu$ is greater than one (denote distinct atoms of $\mu$ by $a_1,a_2,\cdots, a_n\in\R$), then the above proof implies that there is a curve $C_i$ in $\C^+$ that approaches zero nontangentially to $\R$, such that
\begin{align*}
\lim_{z\rightarrow 0, z\in C_i}u_\mu(z)=a_i,\qquad i=1,2,\cdots, n.
\end{align*}
By \cite[Lemma 5.11]{BeVo1993}, the nontangential limit of $u_\mu(z)$ at $0$ is $a_i$ for $i=1,2,\cdots, n$. As the limit is uniquely determined by $\mu$, we have $a_1=a_2=\cdots =a_n$. Therefore, the measure $\mu$ has at most one atom.
\end{proof}

\begin{rem}\label{Bernoulli}
For any $a,b\in\R$ with $a\not=b$ and any $p\in(0,1)$, Theorem \ref{atom} implies that the Bernoulli distribution $p\delta_{a}+(1-p)\delta_{b}$ is not FQID. However, $p\delta_a+(1-p)\delta_b$ is QID for $p\neq 1/2$ (see \cite{LPS}). Therefore, quasi-infinite divisibility does not exhibit complete similarity between the classical and free probability theories.
\end{rem}

\subsection{Convergence of FQID distributions}

We say that a sequence $\{\mu_n\}_n\subset \P$ converges weakly to $\mu\in\P$ if, for all $f\in C_b(\R)$,
\begin{align}\label{W}
\lim_{n\rightarrow\infty}\int_\R f(x)\mu_n(dx)=\int_\R f(x)\mu (dx),
\end{align}
where $C_b(\R)$ is the set of all bounded continuous functions on $\R$. In this case, we write $\mu_n\xrightarrow{w}\mu$. The weak convergence of probability measures can be characterized by Voiculescu transforms (see \cite{BeVo1993,BP96}).

\begin{lem}\label{lem:wconv}
Let $\{\mu_n\}_{n}\subset\P$. Then, the following conditions are equivalent.
\begin{enumerate}[\rm (i)]
\item There exists some $\mu\in\P$ such that $\mu_n\xrightarrow{w}\mu$ as $n\rightarrow\infty$.
\item There exist positive numbers $\lambda$ and $M$ and a function $\varphi$ such that all functions $\varphi_{\mu_n}$ and $\varphi$ are defined on $\Gamma_{\lambda,M}$, and such that
\begin{enumerate}[\rm (a)]
\item $\lim_{n\rightarrow\infty} \varphi_{\mu_n}(z)=\varphi(z)$ uniformly on compact subsets of $\Gamma_{\lambda,M}$;
\item $\sup_{n\in\N}\left|\frac{\varphi_{\mu_n}(z)}{z}\right|\rightarrow 0$ as $|z|\rightarrow\infty$ with $z\in \Gamma_{\lambda,M}$.
\end{enumerate}
\item There exist positive numbers $\lambda$ and $M$ such that all functions $\varphi_{\mu_n}$ are defined on $\Gamma_{\lambda,M}$, and such that
\begin{enumerate}[\rm (a)]
\item $\lim_{n\rightarrow\infty} \varphi_{\mu_n}(iy)$ exists for all $y\ge M$;
\item $\sup_{n\in\N}\left|\frac{\varphi_{\mu_n}(iy)}{y}\right|\rightarrow 0$ as $y\rightarrow\infty$.
\end{enumerate}
\end{enumerate}
In this case, we have $\varphi=\varphi_\mu$ on $\Gamma_{\lambda,M}$.
\end{lem}

A sequence $\{\mu_n\}_n$ of finite (Borel) measures is said to be {\it uniformly bounded} if 
\begin{align}\label{unifbdd}
\sup_{n\in\N} \mu_n(\R)<\infty.
\end{align}
A sequence $\{\mu_n\}_n$ of finite measures is said to be {\it tight} if, for any $\epsilon>0$, there exists a positive number $T>0$ such that
\begin{align}\label{tightness}
\sup_{n\in \N} \mu_n(\R\setminus [-T,T])\le \epsilon.
\end{align}

The above three notions relating to weak convergence, uniform boundedness, and tightness are also defined for a sequence of finite signed measures on $\R$. If $\{\mu_n\}_n$ is a sequence of finite signed measures on $\R$, then we say that \textit{$\mu_n$ converges weakly to some finite signed measure $\mu$} if the measures $\mu_n$ and $\mu$ satisfy condition \eqref{W}, and the family $\{\mu_n\}_n$ is said to be \textit{uniformly bounded} and \textit{tight} if its total variation $|\mu_n|$ satisfies conditions \eqref{unifbdd} and \eqref{tightness}, respectively.

The following characterization of uniform boundedness and tightness is useful.

\begin{lem}\label{lem:tightunif}
We have the following properties.
\begin{enumerate}[\rm (i)]
\item Let $\{\mu_n\}_n$ be a sequence of finite (signed) measures. Then, the sequence $\{\mu_n\}_n$ is tight if and only if every subsequence of $\{\mu_n\}_n$ has a weakly convergent subsequence. 
\item A weakly convergent sequence of finite (signed) measures is uniformly bounded and tight (see \cite[Theorem 8.6.2]{Bo}).
\end{enumerate}
\end{lem}

The class ${\rm ID}(\boxplus)$ is closed with respect to the weak topology. Furthermore, the (weak) convergence of FID distributions is equivalent to that of free characteristic pairs. 

\begin{lem}\label{FIDeq}
We have the following properties.
\begin{enumerate}[\rm (1)]
\item The class ${\rm ID}(\boxplus)$ is weakly closed, that is, if $\{\mu_n\}_n$ is a sequence of FID distributions, $\mu\in\P$, and $\mu_n\xrightarrow{w}\mu$ as $n\rightarrow\infty$, then $\mu$ is also FID.
\item Let $\{\mu_n\}_n$ be a sequence of FID distributions $\mu_n$ with the free characteristic pair $(b_n,\tau_n)$. The following conditions are equivalent.
\begin{enumerate}[\rm (i)]
\item There exists some $\mu\in\P$ such that $\mu_n\xrightarrow{w}\mu$ as $n\rightarrow\infty$; 
\item There exist $b\in\R$ and a finite measure $\tau$ on $\R$ such that $b_n \rightarrow b$ and $\tau_n\xrightarrow{w}\tau$ as $n\rightarrow\infty$. 
\end{enumerate}
In this case, $\mu$ is an FID distribution with the free characteristic pair $(b,\tau)$.
\end{enumerate}
\end{lem}

The following theorem extends the implication $\rm (ii)\Rightarrow (i)$ in Lemma \ref{FIDeq}(2) to FQID distributions. The proof is a simple modification of \cite[Theorem 5.13(ii) $\Rightarrow$ (i)]{BNT06} to $|\tau_n|$. However, we include the proof for readers' convenience.

\begin{thm}\label{prop:converges}
Let $\{\mu_n\}_n$ be a sequence of FQID distributions with free characteristic pairs $(b_n,\tau_n)$. If there exist $b\in \R$ and a finite signed measure $\tau$ on $\R$ such that $b_n\rightarrow b$ and $\tau_n\xrightarrow{w} \tau$, then there exists some $\mu\in\P$ such that $\mu_n\xrightarrow{w}\mu$ as $n\rightarrow\infty$. Moreover, $\mu$ is an FQID distribution with the free characteristic pair $(b,\tau)$.
\end{thm}
\begin{proof}
We verify Lemma \ref{lem:wconv}(iii) to prove this theorem. As $\mu_n$ is FQID, its Voiculescu transform $\varphi_{\mu_n}$ is defined on $\C^+$. For any $y>0$, the map $t\mapsto \frac{1+tiy}{iy-t}$ is continuous and bounded with respect to $t\in\R$.  Hence, the convergence of $\gamma_n$ and the weak convergence of $\tau_n$ imply that
\begin{align*}
\varphi_{\mu_n}(iy)=\gamma_n+\int_\R\frac{1+tiy}{iy-t}\tau_n(dt)\xrightarrow{n\rightarrow\infty} \gamma+\int_\R \frac{1+tiy}{iy-t}\tau(dt), \qquad y>0.
\end{align*}
Therefore, $\lim_{n\rightarrow\infty}\varphi_{\mu_n}(iy)$ exists for all $y>0$. Condition (a) in Lemma \ref{lem:wconv}(iii) holds.

For any $n\in\N$ and $y>0$, we have that
\begin{align*}
\frac{\varphi_{\mu_n}(iy)}{y}=\frac{\gamma_n}{y}+\int_\R \frac{1+tiy}{y(iy-t)}\tau_n(dt).
\end{align*}
Since $\{\gamma_n\}_n$ is a convergent sequence, it is bounded; therefore, it suffices to show that
\begin{align*}
\sup_{n\in\N}\left| \int_\R \frac{1+tiy}{y(iy-t)}\tau_n(dt)\right|\rightarrow 0, \qquad y\rightarrow\infty
\end{align*}
to check condition (b) in Lemma \ref{lem:wconv}(iii). For any $t\in \R$ and $y>0$, we have that
\begin{align*}
\left| \frac{1+tiy}{y(iy-t)}\right| \le \frac{1}{y(y^2+t^2)^{1/2}}+\frac{|t|}{(y^2+t^2)^{1/2}}.
\end{align*}
For any $y\ge 1$, we also have
\begin{align*}
\sup_{t\in\R}\left| \frac{1+tiy}{y(iy-t)}\right|\le 2.
\end{align*}
For any $N\in\N$ and $y\ge 1$, we have that
\begin{align*}
\sup_{t\in[-N,N]}\left| \frac{1+tiy}{y(iy-t)}\right|\le \frac{N+1}{y}.
\end{align*}
Therefore, for any $N\in\N$ and $y\ge 1$, we obtain
\begin{align*}
\sup_{n\in \N}\left| \int_\R \frac{1+tiy}{y(iy-t)}\tau_n(dt)\right| &\le \sup_{n\in\N} \int_\R \left| \frac{1+tiy}{y(iy-t)}\right| |\tau_n|(dt)\\
&\le \frac{N+1}{y} \sup_{n\in \N} |\tau_n|([-N,N])+2\sup_{n\in\N}|\tau_n|(\R\setminus[-N,N])\\
&\le \frac{N+1}{y} \sup_{n\in \N} |\tau_n|(\R)+2\sup_{n\in\N}|\tau_n|(\R\setminus[-N,N]).
\end{align*}

As $\tau_n\xrightarrow{w}\tau$, the sequence $\{|\tau_n|\}_n$ is uniformly bounded and tight by Lemma \ref{lem:tightunif}(ii). Consider an arbitrary $\epsilon>0$. Since $\{|\tau_n|\}_n$ is tight, there exists some $N\in\N$ such that $\sup_{n\in\N}|\tau_n|(\R\setminus[-N,N])< \epsilon/4$.  As $\sup_{n\in\N}|\tau_n|(\R)<\infty$, for any $N\in\N$, there is some $y_0\ge 1$ such that 
\begin{align*}
\frac{N+1}{y}\sup_{n\in \N}|\tau_n|(\R)<\frac{\epsilon}{2}, \qquad y\ge y_0.
\end{align*}
Therefore, if $y\ge y_0$, then
\begin{align*}
\sup_{n\in \N}\left| \int_\R \frac{1+tiy}{y(iy-t)}\tau_n(dt)\right|&\le \frac{N+1}{y} \sup_{n\in \N} |\tau_n|(\R)+2\sup_{n\in\N}|\tau_n|(\R\setminus[-N,N])<\epsilon.
\end{align*}

Finally, Lemma \ref{lem:wconv}(iii) holds. Therefore, there is a probability measure $\mu$ on $\R$ such that $\mu_n\xrightarrow{w} \mu$ as $n\rightarrow\infty$, and
\begin{align*}
\varphi_\mu(z)=b+\int_\R \frac{1+tz}{z-t}\tau(dt), \qquad z\in \C^+.
\end{align*}
Hence, $\mu$ is an FQID distribution with the free characteristic pair $(b,\tau)$.
\end{proof}

\begin{rem}
Whether the converse implication of Theorem \ref{prop:converges} is true remains an open question. More precisely, we wish to know whether, if $\mu_n\xrightarrow{w}\mu$, then $b_n\rightarrow b$ and $\tau_n\rightarrow\tau$ as $n\rightarrow\infty$. According to the proof in \cite{BNT06}, if $\mu_n$ is an FID distribution with the free characteristic pair $(b_n,\tau_n)$, then the positivity of $\tau_n$ can be used to show the implication (i) $\Rightarrow$ (ii) in  Lemma \ref{FIDeq}(2). However, if $\mu_n$ is FQID, then $\tau_n$ is a signed measure, and we cannot apply the techniques used in \cite{BNT06}.
\end{rem}

One can obtain the converse of Theorem \ref{prop:converges} under some additional assumptions regarding free quasi-generating pairs.

\begin{thm}\label{thm:F}
Let $\{\mu_n\}_n$ be a sequence of FQID distributions $\mu_n$ with free characteristic pairs $(b_n,\tau_n)$ and let $\mu$ be an FQID distribution with the free characteristic pair $(b,\tau)$. Suppose that $\tau_n=\tau_n^+-\tau_n^-$ for each $n\in\N$, where $\tau_n^{\pm}$ are the positive and negative parts of $\tau$. If $\mu_n\xrightarrow{w}\mu$ as $n\rightarrow\infty$, and $\{\tau_n^-\}_n$ is tight and uniformly bounded, then $b_n\rightarrow b$ and $\tau_n\xrightarrow{w}\tau$ as $n\rightarrow\infty$. 
\end{thm}

To prove this theorem, we investigate convergence, tightness, and uniform boundedness for free characteristic pairs. Note that the proofs of Lemma \ref{lem:F} and Theorem \ref{thm:F} are obtained by the direct application of the results in \cite{LPS}.

\begin{lem}\label{lem:F}
Let $\{\mu_n\}_n$ be a sequence of FQID distributions with free characteristic pairs $(b_n,\tau_n)$. If the sequence $\{\mu_n\}_n$ is tight and the sequence $\{\tau_n^-\}_n$ is tight and uniformly bounded, then the sequence $\{b_n\}_n$ is bounded and the sequences $\{\tau_n^+\}_n$ and $\{|\tau_n|\}_n$ are tight and uniformly bounded.
\end{lem}

\begin{proof}
As the sequence $\{\mu_{n}\}_n$ is tight and $\{\tau_n^-\}_n$ is tight and uniformly bounded, for any subsequences $\{k_n\}_n$ in $\N$, the subsequence $\{\mu_{k_n}\}_n$ is also tight and the subsequence $\{\tau_{k_n}^-\}_n$ is also tight and uniformly bounded. By Lemma \ref{lem:tightunif}(i), there exists a further subsequence $\{l_{k_n}\}_n$ of the sequence $\{k_n\}_n$ such that the subsequences $\{\mu_{l_{k_n}}\}_n$ and $\{\tau_{l_{k_n}}^-\}_n$ converge weakly to some probability measure $\mu$  and some finite measure $\tau^-$, respectively. 

Let $\rho_{l_{k_n}}$ be an FID distribution with the free characteristic pair $(0,\tau_{l_{k_n}}^-)$ for each $n$. By Lemma \ref{FIDeq}, there is an FID distribution $\rho$  with the free characteristic pair $(0,\tau^-)$ such that $\rho_{l_{k_n}}\xrightarrow{w} \rho$ as $n\rightarrow\infty$.  Thus, we have that $\mu_{l_{k_n}}\boxplus \rho_{l_{k_n}}\xrightarrow{w} \mu\boxplus \rho$ as $n\rightarrow\infty$. Since $\mu_{l_{k_n}}\boxplus \rho_{l_{k_n}}$ is FID with the free characteristic pair $(b_{l_{k_n}}, \tau_{l_{k_n}}^+)$ and Lemma \ref{FIDeq}(1) holds, the measure $\mu\boxplus \rho$ is also FID. Therefore, the sequence $\{\tau_{l_{k_n}}^+\}_n$ converges weakly and the sequence $\{b_{l_{k_n}}\}_n$ converges as $n\rightarrow\infty$ by Lemma \ref{FIDeq}(2).

To summarize the above discussion, for any subsequence $\{k_n\}_n$ in $\N$, there is a further subsequence $\{l_{k_n}\}_n$ of the sequence $\{k_n\}_n$ such that the measure $\tau_{l_{k_n}}^+$ converges weakly and the sequence $\{b_{l_{k_n}}\}_n$ converges. Consequently, the sequence $\{b_n\}_n$ is bounded and the sequence $\{\tau_n^+\}_n$ is tight and uniformly bounded; therefore, so is the sequence $\{|\tau_n|\}_n$.
\end{proof}

\begin{proof}[Proof of Theorem \ref{thm:F}]
By our assumption, the sequence $\{\mu_n\}_n$ is tight, and therefore the sequence $\{b_n\}_n$ is bounded and the sequence $\{\tau_n^+\}_n$ is tight and uniformly bounded by Lemma \ref{lem:F}. Thus, $\{\tau_n\}_n$ is also tight and uniformly bounded. If either (or both) of the sequences $\{b_n\}_n$ and $\{\tau_n\}_n$ does not satisfy the statement of this theorem, then the tightness and (uniform) boundedness of these sequences imply that the sequence $\{\mu_n\}_n$ does not converge weakly, which contradicts the weak convergence of the sequence $\{\mu_n\}_n$. Hence, there exist $b^*\in\R$ and a finite signed measure $\tau^*$ such that $b_n\rightarrow b^*$ and $\tau_n\xrightarrow{w}\tau^*$ as $n\rightarrow\infty$. By the uniqueness of free characteristic pairs (see Proposition \ref{uniquegenerating}) and the convergence $\mu_n\xrightarrow{w}\mu$, we have that $b=b^*$ and $\tau=\tau^*$.
\end{proof}

\begin{problem}
In the classical case, the class of all QID distributions is not closed under weak convergence---see \cite[Section 4]{LPS} for a concrete example using the Bernoulli distributions.
In the free setting, this example does not work well because the Bernoulli distributions are not FQID. Although we will give many examples of FQIDs in Section 4, they are closed under weak convergence.
Thus, whether the class of all FQIDs is closed under weak convergence remains an open problem.
\end{problem}

\subsection{Free L\'{e}vy-Khintchine type representations of FQID distributions}

In this section, we describe how $R$-transforms of FQID distributions give the free L\'{e}vy-Khintchine representations.

\begin{prop}\label{thm:FQID}
Let $\mu$ be an FQID distribution with a free characteristic pair $(b,\tau)$. Then, we have
\begin{align}\label{eq:R}
R_\mu(z)=az^2+  \gamma z+ \int_\R \left(\frac{1}{1-xz}-1-xz\mathbf{1}_{[-1,1]}(x)\right) \nu(dx),\qquad z\in \C^-,
\end{align}
for some $\gamma,a\in\R$ and a quasi-L\'{e}vy-type measure $\nu$ on $\R$. Conversely, if the FQID distribution $\mu$ has an $R$-transform given by \eqref{eq:R}, then $\mu$ has a free characteristic pair $(b,\tau)$ with the formula \eqref{Triplet}. 
\end{prop}


A quasi-L\'{e}vy-type measure $\nu$ is called a {\it free quasi-L\'{e}vy measure} for the FQID distribution $\mu$ if the $R$-transform of $\mu$ has the representation given by \eqref{eq:R}. In this case, the triplet $(a,\nu,\gamma)$ is called a {\it free characteristic triplet} of the FQID distribution $\mu$.

\begin{proof}[Proof of Proposition \ref{thm:FQID}]
We obtain the form given by \eqref{eq:R} for the $R$-transform $R_\mu$, and show that \eqref{Triplet} describes the free characteristic pair. Note that $a=\tau(\{0\})\ge 0$ is not necessary, because $\tau$ is a signed measure. We only have to show that $\nu$ is a quasi-L\'{e}vy-type measure.

For each $r>0$, we have that
\begin{align*}
\nu|_{\calB_r} (\R\setminus(-r,r))=\int_{\R\setminus (-r,r)}\frac{1+x^2}{x^2}\tau(dx) \le \frac{1+r^2}{r^2}\tau(\R\setminus (-r,r))<\infty,
\end{align*}
and so $\nu|_{\calB_r}$ is finite. Moreover, we have
\begin{align*}
\int_\R (1\land x^2)|\nu|(dx)&=\int_{-1}^1 x^2\cdot \frac{1+x^2}{x^2} |\tau|(dx)+\int_{\R\setminus (-1,1)} |\nu|(dx)\\
&\le 2|\tau|((-1,1))+|\nu|_{\calB_1}|(\R\setminus (-1,1)) <\infty.
\end{align*}
Therefore, $\nu$ is a quasi-L\'{e}vy-type measure on $\R$. 
\end{proof}

Note that the kernel function of the $R$-transform of an FQID distribution is integrable with respect to the quasi-L\'{e}vy-type measure.

\begin{prop}\label{prop:integrable}
For $z\in\C^-$, the function on $\R$ defined by
\begin{align*}
x\mapsto \frac{1}{1-zx}-1-zx\mathbf{1}_{[-1,1]}(x)=:k(x,z)
\end{align*}
is integrable with respect to a quasi-L\'{e}vy-type measure $\nu$.
\end{prop}
\begin{proof}
Let $\nu^+$ and $\nu^-$ be the positive and negative parts of $\nu$, respectively. Then,
\begin{align*}
\int_\R k(x,z)\nu^+(dx)&=z^2\int_{-1}^1 \frac{x^2}{1-zx} \nu^+(dx)+z\int_{\R\setminus[-1,1]} \frac{x}{1-zx}\nu^+(dx)\\
&=z^2\int_{-1}^1 \frac{1}{1-zx} (1\land x^2) \nu^+(dx)+z\int_{\R\setminus[-1,1]} \frac{x}{1-zx}\nu^+|_{\calB_1}(dx).
\end{align*}

We show that the first term on the right-hand side of the above equality is integrable. For $z=u+iv\in \C^-$ and $x\in[-1,1]$, we obtain
\begin{align*}
\left|\frac{1}{1-zx} \right|^2 =\frac{1}{(u^2+v^2)x^2-2ux+1}\le \frac{u^2+v^2}{v^2}=\left(\frac{|z|}{\Im (z)}\right)^2.
\end{align*}
As $\int_\R(1\land x^2)\nu^+(dx)<\infty$, we have that
\begin{align*}
\left| z^2\int_{-1}^1 \frac{1}{1-zx} (1\land x^2) \nu^+(dx)\right| \le \frac{|z|^3}{|\Im(z)|} \int_{-1}^1 (1\land x^2)\nu^+(dx)<\infty , \qquad z\in \C^-.
\end{align*}
 
 Next, we prove that the second term on the right-hand side of the above equality is also integrable. For $z=u+iv\in \C^-$ and $x\in\R$, we have
 \begin{align*}
 \left|\frac{x}{1-zx} \right|^2 =\frac{|x|^2}{(1-ux)^2+v^2x^2}\le \frac{1}{v^2}=\frac{1}{\Im (z)^2}.
 \end{align*}
 As $\nu^+|_{\calB_1}$ is finite, we obtain
 \begin{align*}
\left| z\int_{\R\setminus[-1,1]} \frac{x}{1-zx}\nu^+|_{\calB_1}(dx)\right| \le \left|\frac{z}{\Im(z)}\right| \nu^+|_{\mathcal{B}_1}(\R\setminus [-1,1])<\infty, \qquad z\in \C^-.
 \end{align*}
 
Therefore, the function $k(\cdot, z)$ is integrable with respect to $\nu^+$. Similarly, it is also integrable with respect to $\nu^-$, and hence is so with respect to $\nu$.
\end{proof}

If $\mu$ is FQID with the free characteristic triplet $(a,\nu,\eta)$, then 
\begin{align*}
R_{\mu}(z)=az^2+ \eta z+\int_\R k(x,z)\nu^+(dx)- \int_\R k(x,z)\nu^-(dx), \qquad z\in\C^-,
\end{align*}
by Proposition \ref{prop:integrable}. It follows that there exist FID distributions $\mu^+$ and $\mu^-$ with free characteristic triplets $(a,\nu^+,\gamma)$ and $(0,\nu^-,0)$, respectively, such that $\mu^-\boxplus \mu=\mu^+$.

\subsection{FQID distributions with compact support}
We say that a sequence $\{s_{n}\}_{n=0}^{\infty}$ of complex numbers {\it grows at most exponentially} if there exists a constant $c>0$ such that 
\begin{align*}
|s_{n}| \leq c^{n}, \quad \text{ for all }n\in\N.
\end{align*}
A finite measure $\rho$ has a compact support if and only if the $n$-th moment $m_n(\rho)$ of $\rho$ exists for all $n\in \N$ and the sequence $\{m_n(\rho)\}_{n=0}^\infty$ grows at most exponentially (see \cite[Lemma 13.13]{NiSpBook}).

\begin{thm}\label{thm:Cpt}
Let $\mu$ be an FQID distribution with the free characteristic pair $(b,\tau)$. 
If the signed finite measure $\tau$ is compactly supported, then so is the measure $\mu$.
Furthermore, we have that $\tau(\R)\ge -m_2(\tau)$.
\end{thm}

\begin{proof}
From the relation $R_\mu(z)=z\varphi_\mu(1/z)$, for $z\in \C^-$ where $|z|$ is sufficiently small, we obtain
\begin{align*}
R_\mu(z)&=bz+z\int_\R \frac{z+x}{1-zx}\tau(dx)\\
&=(b+m_1(\tau))z+(\tau(\R)+m_2(\tau))z^2+\sum_{n=3}^\infty (m_{n-2}(\tau)+m_n(\tau))z^n.
\end{align*}
Define a sequence $\{\kappa_n(\mu)\}_n$ as follows: $\kappa_1(\mu):=b+m_1(\tau)$, $\kappa_2(\mu):=\tau(\R)+m_2(\tau)$, and $\kappa_n(\mu):=m_{n-2}(\tau)+m_n(\tau)$ for all $n\ge 3$. Then $\kappa_n(\mu)$ is the $n$-th free cumulant of $\mu$ (see \cite{NiSpBook} for further information).

By the Hahn-Jordan decomposition theorem, we can find two finite measures $\tau^+$ and $\tau^-$ such that $\tau=\tau^+-\tau^-$. As $\tau$ has a compact support, so do $\tau^+$ and $\tau^-$. By \cite[Lemma 13.13]{NiSpBook}, there exist $c_+,c_->0$ such that $|m_n(\tau^+)|\le c_+^n$ and $|m_n(\tau^-)|\le c_-^n$ for all $n\in \N$. Therefore, we obtain
\begin{align*}
|\kappa_1(\mu)|&\le |b|+|m_1(\tau)|\le |b|+|m_1(\tau_+)|+|m_1(\tau_-)|\le |b|+c_++c_-,\\
|\kappa_2(\mu)|&\le |\tau(\R)|+|m_2(\tau)|\le |\tau(\R)|+c_+^2+c_-^2,\\
|\kappa_n(\mu)|&\le |m_{n-2}(\tau)|+|m_n(\tau)|\le c_+^{n-2}+c_-^{n-2}+c_+^n+c_-^n, \qquad n\ge 3.
\end{align*}
Then, there exists some $c>0$ such that $|\kappa_n(\mu)|\le c^n$ for all $n\in \N$, and therefore, the sequence $\{\kappa_n(\mu)\}_{n=1}^\infty$ grows at most exponentially. By \cite[Proposition 13.15]{NiSpBook}, the moment sequence $\{m_n(\mu)\}_{n=0}^\infty$ also grows at most exponentially. Hence, $\mu$ is compactly supported.

In this case, we then obtain $\tau(\R)+m_2(\tau)=\kappa_2(\mu)=m_2(\mu)-m_1(\mu)^2\ge0$.
\end{proof}

\begin{problem}
Suppose that $\mu$ is an FQID distribution with the free characteristic pair $(b,\tau)$. If $\mu$ has a compact support, then does the measure $\tau$ have a compact support?
\end{problem}

\section{Examples}

Let $\R_+:=[0,\infty)$. In the following, the principal square root of $z$ is defined by
\begin{equation}
\label{square-root}
\sqrt{z} :=|z|^{\frac{1}{2}}e^{\frac{i \arg(z)}{2}},\quad 0 \le \arg(z) < 2\pi,
\end{equation}
namely, the square root function is defined using the nonnegative real axis $\R_{+}$ as a branch cut.
Note that, in this case, $z$ has to lie in $\C\setminus\R_{+}$.

In this section, various examples of FQID distributions (or free quasi-characteristic triplets) and their distributional properties are provided.
For this purpose, let us recall three important families of distributions
in free probability (see \cite{BeVo1993, NiSpBook, MS17} for details).

\begin{itemize}
\item Let $\mathbf{C}_a$ be the Cauchy distribution with the parameter $a>0$, that is, 
\begin{align*}
\mathbf{C}_a(\rmd x):=\frac{a}{\pi}\frac{1}{x^2+a^2}\mathbf{1}_\mathbb{R}(x)\rmd x.
\end{align*}
The special case $a=1$, the standard Cauchy distribution $\mathbf{C}_{1}$, is denoted by $\mathbf{C}$.
Note that $\mathbf{C}_{a} = \mathbf{D}_{a}(\mathbf{C})=\mathbf{C}^{\boxplus a}$ for all $a>0$.
The Cauchy distribution is FID and  
its Voiculescu transform is given by
\begin{align*}
\varphi_{\mathbf{C}_{a}}(z)=\int_{\R} \frac{1 +zt}{z-t}\frac{a}{\pi(1+t^{2})}\rmd t =-ai, \qquad z\in\mathbb{C}^+.
\end{align*}
In this case, the free characteristic triplet is $(0,\frac{a}{\pi u^2}\rmd u,0)$.

\item Let ${\bf MP}(\lambda)$ be the Marchenko-Pastur law with the parameter (variance) $\lambda>0$.
The Voiculescu transform $\varphi_{{\bf MP}(\lambda)}$ is given by
\begin{align*}
\varphi_{{\bf MP}(\lambda)}(z)
=\frac{\lambda}{2} + \int_{\R}\frac{1+tz}{z-t} \frac{\lambda}{2}\delta_{1}(\rmd t)
=\frac{\lambda z}{z-1}, \qquad z\in\mathbb{C}^+.
\end{align*}
The measure ${\bf MP}(\lambda)$ is FID and the free characteristic triplet is $(0,\lambda \delta_{1},0)$.
Note that, for $c\not=0$, the free characteristic triplet of $\mathbf{D}_{c}({\bf MP}(\lambda))$ is $(0,\lambda \delta_{c},0)$ and the Voiculescu transform $\varphi_{\mathbf{D}_{c}({\bf MP}(\lambda))}$ is given by
\begin{align*}
\varphi_{\mathbf{D}_{c}({\bf MP}(\lambda))}(z)
=\frac{\lambda}{c^{2}+1} + \int_{\R}\frac{1+tz}{z-t} \frac{\lambda c}{c^{2}+1}\delta_{c}(\rmd t)
=\frac{\lambda c z}{z-c}.
\end{align*}

\item Let ${\bf S}(m,\sigma^2)$ be the semicircle law with mean $m\in\R$ and variance $\sigma^2>0$, i.e.,
\begin{align*}
{\bf S}(m,\sigma^2)(dx):= \frac{1}{2\pi \sigma^2} \sqrt{4\sigma^2- (x-m)^2} \mathbf{1}_{[m-2\sigma,m+2\sigma]}(x)dx.
\end{align*}
The special case ${\bf S}(0,1)$ is simply denoted by ${\bf S}$. This is FID and the Voiculescu transform $\varphi_{{\bf S}(m,\sigma^{2})}$ has the form
\begin{align*}
\varphi_{{\bf S}(m,\sigma^{2})}(z) 
= m + \int_{\R} \frac{1 +zt}{z-t}\sigma^{2}\delta_{0}(\rmd t)
= m +  \frac{\sigma^{2}}{z}, \qquad z\in\mathbb{C}^+.
\end{align*}
Its free characteristic triplet is $(\sigma^{2},0,m)$. Note that $\mathbf{D}_{a}({\bf S}(m,\sigma^{2}))= {\bf S}(am,a^{2}\sigma^{2})$.

\end{itemize}

\subsection{Free deconvolution with the Cauchy distribution}

In this section, we construct explicit FQID distributions from the Cauchy distribution. 
The support of the free L{\'e}vy measure of the Cauchy distribution is unbounded, whereas the support of the free L{\'e}vy measure of compactly supported FID distributions is bounded. Thus, the difference between the free L{\'e}vy measure of the Cauchy distribution and that of a compactly supported FID distribution becomes a signed measure. 
We investigate whether there exist FQID probability measures whose quasi-free L{\'e}vy measure is a signed measure. 

This idea comes from the results in \cite{ATV}.
If $\mu\in\P$ has a finite variance $\sigma^2$, then by \cite[Theorem 1.2]{ATV}, there exists some $\rho(\mu) \in \P$ such that 
\begin{align}\label{FQIDex}
\mu\boxplus \rho(\mu)=\mathbf{C}_{2\sqrt{2}\sigma}.
\end{align}
In other words, for a given  $\mu\in\P$ having a finite variance $\sigma^2$, $\rho(\mu) := \mathbf{C}_{2\sqrt{2}\sigma}\boxminus \mu$ defines a probability measure on $\R$. 
Furthermore, if $\mu$ is an FID distribution with a compact support, then $\rho(\mu)$ is FQID. The probability measures
$\mathbf{MP}(\lambda)$ and $\mathbf{S}(0,\sigma^2)$ are examples of such $\mu$, as both distributions have finite variances and are FID with compact supports. That is,
\begin{itemize}
\item $\mathbf{MP}(\lambda)\boxplus \rho(\mathbf{MP}(\lambda))=\mathbf{C}_{2\sqrt{2\lambda}}\,$ and $\,\rho(\mathbf{MP}(\lambda))$ is FQID,
\item$\mathbf{S}(0,\sigma^2)\boxplus \rho(\mathbf{S}(0,\sigma^2))=\mathbf{C}_{2\sqrt{2}\sigma}\,$ and $\,\rho(\mathbf{S}(0,\sigma^2))$ is FQID.
\end{itemize}

In Sections 4.1.1, 4.1.2, and 4.1.3, we investigate  signed measures generated from the L{\'e}vy measure of the Cauchy distribution and the Marchenko-Pastur, semicircle, and free Meixner distributions, respectively.

\subsubsection{Case of the Marchenko-Pastur law}

From now on we will denote $\mathbf{MP}\big(c,\lambda):=\mathbf{D}_{c}\left(\mathbf{MP}(\lambda)\right)$.
Let $a,\lambda>0$ and $c \not= 0$.
Consider the analytic function
\begin{align}\label{eq:VCMP}
\varphi(z) 
= \varphi_{\mathbf{C}_{a}}(z) - \varphi_{\mathbf{MP}(c,\lambda)}(z)
= -ai -\frac{\lambda c z}{z- c}, \qquad z\in\C^+.
\end{align}

We will prove that $\varphi$ is the Voiculescu transform of some probability measure, or equivalently, that the function $K:\C^+\rightarrow \C^+$ defined by 
\begin{align*}
K(z) :=\varphi(z)+z= -ai -\frac{\lambda c z}{z- c}+z, \qquad z\in\C^+
\end{align*}
is the compositional right inverse of an $F$-transform of some probability measure.

A straightforward calculation gives
\begin{equation}
\label{F-trans-4.1}
F(z) :=K^{-1}(z)
= \frac{z+c(\lambda+1)+ai + \sqrt{q(z)}}{2},\quad z\in\C^+,
\end{equation}
with an appropriate branch of the square root, where 
$$
q(z) :=(z+c(\lambda+1)+ai)^{2} -4c(z+ai).
$$
We will show that $F$ is a Pick function.
It is easy to check that $\lim_{y\rightarrow\infty} F(iy)/iy=1$.
Hence, it suffices to show that $q(\C^{+}) \subset \C\setminus\R_{+}$ (recall that the branch cut of the square root is $\R_{+}$ in our setting). Then, $\sqrt{q(z)} \in \C^{+}$ is well-defined on $\C^{+}$, which yields $\Im F(z) >0$ for all $z \in \C^{+}$.

Let $z =x + iy\in\C^{+}$ (namely $x \in \R$ and $y > 0)$. We have that
\begin{equation*}
q(x+iy)= (x + c(\lambda+1))^{2} -4cx -(y+a)^{2} + 2i(y+a)(x + c(\lambda-1)).
\end{equation*}
Then, $q(x+iy) \in \R$ if and only if $x = c(1 - \lambda)$. From the equation
\begin{align*}
q(c(1 - \lambda)+iy)= 4c^{2}\lambda -(y+a)^{2},
\end{align*}
it follows that, if $0<\lambda \leqq \left(\frac{a}{2c}\right)^{2}$, then $q(x+iy) \notin \R_{+}$, which implies that $F(\C^+)\subset \C^+$. 
We conclude that, in this case, $F$ in \eqref{F-trans-4.1} is a Pick function.

By Proposition \ref{exist}, there exists a probability measure $\rho_{a,c.\lambda}$ on $\R$ such that $F=F_{\rho_{a,c,\lambda}}$. Finally, we obtain
\begin{align*}
\mathbf{MP}(c,\lambda)\boxplus \rho_{a,c,\lambda}=\mathbf{C}_a,
\end{align*}
and so $\rho_{a,c,\lambda}$ is FQID for all $a,\lambda>0$ and $c\neq 0$ with $0<\lambda \leqq \left(\frac{a}{2c}\right)^{2}$.

We further observe that, if $0<y<\frac{c^2\lambda}{a}$, 
\begin{align*}
\Im \varphi_{\rho_{a,c,\lambda}}(c+iy)=-a+\frac{c^2\lambda}{y}>0.
\end{align*}
Hence, the measure $\rho_{a,c,\lambda}$ is not FID. 
For all $x\in \R$, calculus indicates that the nontangential limit $\sphericalangle{\rm \mathchar`-}\lim_{z\rightarrow x} F_{\rho_{a,c,\lambda}}(z)$ is not zero. Thus, $\rho_{a,c,\lambda}$ has no singular parts by \cite[Theorem 2.7]{P03}. 
We summarize the results.

\begin{thm}\label{thm:ex1}
Let $a,\lambda>0$, $c\not=0$ with $0 < \lambda \leqq \left(\frac{a}{2c}\right)^{2}$.
Then, there exists some $\rho_{a,c,\lambda}\in\P$ such that its Voiculescu transform is
\begin{align*}
\varphi_{\rho_{a,c,\lambda}}(z) = -ai -\frac{\lambda c z}{z- c}, \qquad z\in \C^+.
\end{align*}
It has the following properties:
\begin{enumerate}
\item $ \mathbf{MP}(c,\lambda) \boxplus \rho_{a,c,\lambda}=  \mathbf{C}_{a} $. In other words, $\rho_{a,c,\lambda}=  \mathbf{C}_{a}\boxminus \mathbf{MP}(c,\lambda)$ is FQID, rather than FID. Moreover, it is absolutely continuous with respect to the Lebesgue measure.
\item Its free characteristic triplet is given by $(0,\frac{a\rmd u}{\pi u^2}-\lambda\delta_c,0)$. 
\item Its free characteristic pair is given by $(-\frac{\lambda}{c^2+1},\frac{a \rmd u}{\pi(u^2+1)}-\frac{\lambda c}{c^2+1}\delta_c)$.
\end{enumerate}
\end{thm}

Let us consider two specific cases.

If $a=2\sqrt{2\lambda}$ and $c=1$, then $\rho_{2\sqrt{2\lambda}, 1, \lambda}$ coincides with the measure $\rho(\mathbf{MP}(\lambda))$ in \eqref{FQIDex}. If $a=0$, the function $\varphi$ in \eqref{eq:VCMP} is not the Voiculescu transform.

If $c=a>0$, then for $0<\lambda \leqq \frac{1}{4}$,
there exists a probability measure $\rho_{a,\lambda}:=\rho_{a,a,\lambda}\in\P$ such that $\varphi_{\rho_{a,\lambda}}(z) = -ai -\frac{\lambda a z}{z- a}$ by the above discussion. We can see that $\mathbf{D}_{a}(\rho_{1,\lambda})=\rho_{a,\lambda}$.

\begin{ex}
We consider the Cauchy transform of $\rho_{1,\lambda}$:
\begin{align*}
G_{\rho_{1,\lambda}}(z) &= 1/F_{\rho_{1,\lambda}}(z)  \\
&= \frac{2}{z+(\lambda+1)+i + \sqrt{(z+(\lambda+1)+i)^{2} -4(z+i)}}\\
&= \frac{z+(\lambda+1)+i - \sqrt{(z+(\lambda+1)+i)^{2} -4(z+i)}}{2(z+i)}\\
&= \frac{z^{2}+(\lambda+1)(z-i)+1  - (z-i)\sqrt{(z+(\lambda+1)+i)^{2} -4(z+i)}}{2(z^{2} + 1)},
\end{align*}
for $z\in\C^+$. By the Stieltjes inversion formula, the probability density function $f(x)$ of $\rho_{1,\lambda}$ is
\begin{align*}
f(x) &= -\frac{1}{\pi}\lim_{y\downarrow 0}\Im(G_{\rho_{1,\lambda}}(x+iy))= \frac{1}{2\pi(x^{2} + 1)} \left\{\lambda + 1 + xv(x)- u(x) \right\}, \qquad x\in\R,
\end{align*}
where
\begin{align*}
u(x) =\Re\left[\sqrt{(x+(\lambda+1)+i)^{2} -4(x+i)}\right], \qquad x\in\R
\end{align*}
and
\begin{align*}
v(x) = \Im\left[\sqrt{(x+(\lambda+1)+i)^{2} -4(x+i)}\right], \qquad x\in\R.
\end{align*}
We give an explicit density function for $\lambda = 1/4$ below. Set
\begin{align*}
q(x):=\left(x+\left(\frac{1}{4}+1\right)+i\right)^2-4(x+i)=\frac{1}{16}(4x-3)^2+\frac{1}{2}(4x-3)i, \qquad x\in\R.
\end{align*}
We then obtain $q(x)=re^{i\theta}$, where
\begin{align*}
r=|q(x)|=\frac{1}{16}|4x-3|\sqrt{16x^2-24x+73}
\end{align*}
and
\begin{align*}
\theta&=\arg(q(x))=\begin{cases}
\arcsin \left( \frac{8}{\sqrt{16x^2-24x+73}}\right), & x\ge 3/4\\
2\pi-\arcsin \left( \frac{8}{\sqrt{16x^2-24x+73}}\right), & x< 3/4.
\end{cases}
\end{align*}
Using calculus, we find that
\begin{align*} 
u(x)&=\Re \sqrt{q(x)}=\sqrt{r}\cos \left(\frac{1}{2} \theta\right)\\
&=\frac{1}{4} \text{sign}(4x-3)\sqrt{|4x-3|}(16x^2-24x+73)^{1/4} \cos\left(\frac{1}{2} \arcsin \left( \frac{8}{\sqrt{16x^2-24x+73}}\right) \right)\\
&=\frac{1}{4\sqrt{2}} \text{sign}(4x-3)\sqrt{|4x-3|} \sqrt{\sqrt{16x^2-24x+73}+|4x-3|}
\end{align*}
and
\begin{align*}
v(x)&=\Im \sqrt{q(x)}=\sqrt{r}\sin \left(\frac{1}{2}\theta\right)\\
&=\frac{1}{4} \sqrt{|4x-3|}(16x^2-24x+73)^{1/4} \sin\left( \frac{1}{2} \arcsin \left( \frac{8}{\sqrt{16x^2-24x+73}}\right)\right)\\
&=\frac{1}{4\sqrt{2}}\sqrt{|4x-3|}\sqrt{\sqrt{16x^2-24+73}-|4x-3|},
\end{align*}
where 
\begin{align*}
\text{sign}(x)=\begin{cases}
1, &x\ge 0\\
-1,& x<0.
\end{cases}
\end{align*}
Finally, we obtain the density function of $\rho_{1,1/4}$ (see Figure \ref{fig:one}):
\begin{align*}
f(x)&=\frac{1}{2\pi (x^2+1)}\left( \frac{5}{4}+xv(x)-u(x)\right)\\
&=\frac{5\sqrt{2}+\sqrt{|4x-3|}\left( x p_-(x)- \text{sign}\left(4x-3\right) p_+(x)\right) }{8\sqrt{2}\pi (x^2+1)}, \qquad x\in\R,
\end{align*}
where
\begin{align*}
p_{\pm}(x):=\sqrt{\sqrt{16x^2-24x+73}\pm |4x-3|}, \qquad x\in\R.
\end{align*}

\begin{figure}[htbp]
 \begin{center}
  \includegraphics[width=100mm]{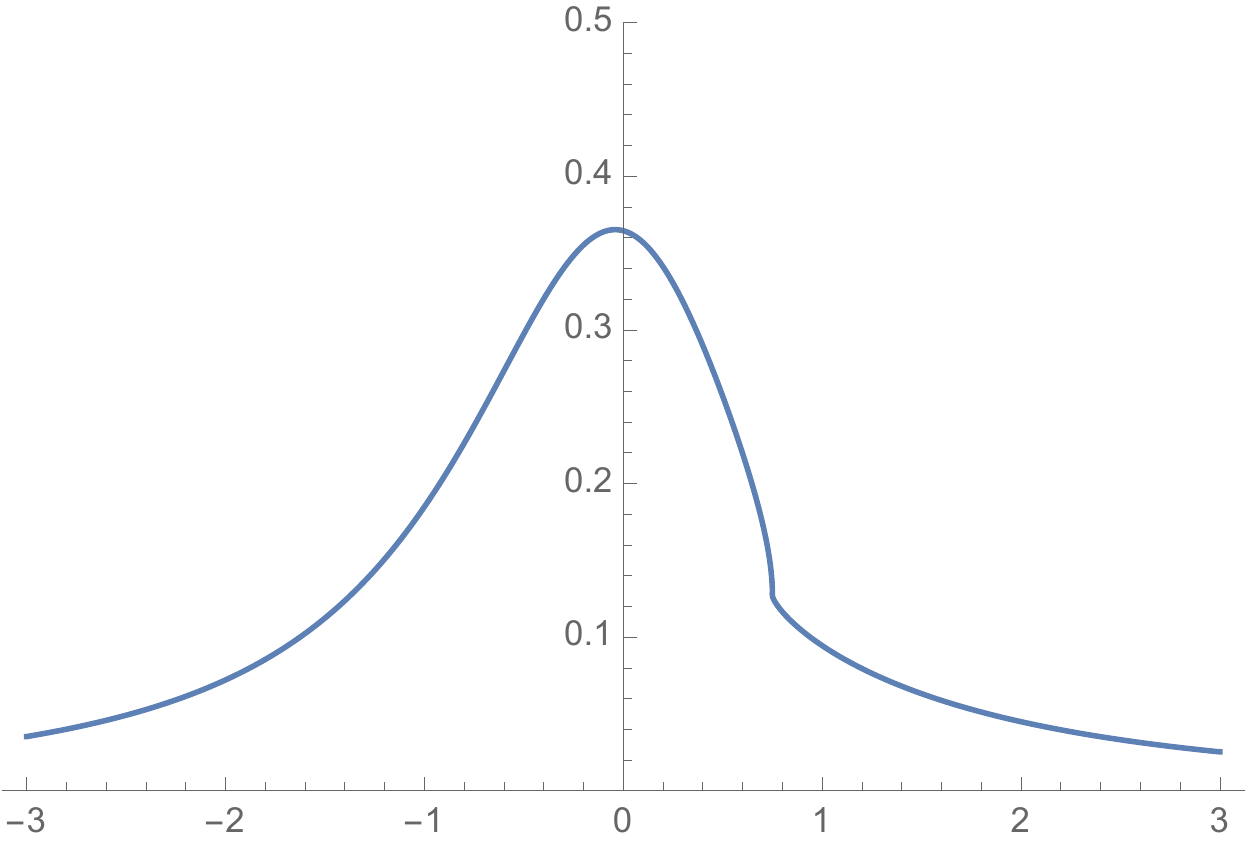}
 \end{center}
 \caption{Probability density function of $\rho_{1,1/4}$}
 \label{fig:one}
\end{figure}

\end{ex}

\subsubsection{Case of the semicircle law}

Next, consider 
$$
\varphi(z) = \varphi_{\mathbf{C}_{a}}(z) - \varphi_{\mathbf{S}(0,\sigma^{2})}(z) = -ai -\frac{\sigma^{2}}{z}
$$ 
for $a,\sigma>0$ and $z\in\C^+$.
We will show that $\varphi$ is the Voiculescu transform of some probability measure.

Let $K(z) := z -ai -\frac{\sigma^{2}}{z}$. Then, a formal computation implies that
\begin{align*}
F(z) :=K^{-1}(z)= \frac{z+ai + \sqrt{(z+ai)^{2} +4\sigma^{2}}}{2}, \qquad z\in\C^+.
\end{align*}
It is easy to see that $F(iy)/iy\rightarrow1$ as $y\rightarrow\infty$. For $2\sigma \le a$,  $\sqrt{(z+ai)^2+4\sigma^2}$ is in $\C^+$ for all $z\in \C^+$.
In a similar manner as for Example \ref{thm:ex1}, there is a probability measure $\gamma_{a,\sigma^2}$ such that $F=F_{\gamma_{a,\sigma^2}}$, and therefore $\varphi=\varphi_{\gamma_{a,\sigma^2}}$ for all $a,\sigma>0$ with $2\sigma\le a$. By definition, it satisfies 
$ \mathbf{S}(0,\sigma^{2})\boxplus \gamma_{a,\sigma^{2}} =\mathbf{C}_{a} $, where $\gamma_{a,\sigma^2}$ is not FID. By \cite[Theorem 2.7]{P03}, the measure $\gamma_{a,\sigma^2}$ is absolutely continuous with respect to the Lebesgue measure.
We summarize the above facts as follows.

\begin{thm}\label{C-S}
Let $a,\sigma>0$ with $2\sigma\le a$. 
Then, there exists some $\gamma_{a,\sigma}\in\P$ for which the Voiculescu transform is
\begin{align*}
\varphi_{\gamma_{a,\sigma^2}}(z)=-ai-\frac{\sigma^2}{z}, \qquad z\in\C^+.
\end{align*}
It has the following properties:
\begin{enumerate}
\item $ \mathbf{S}(0,\sigma^{2})\boxplus \gamma_{a,\sigma^{2}} =\mathbf{C}_{a} $. In other words, $\gamma_{a,\sigma^{2}} =\mathbf{C}_{a} \boxminus \mathbf{S}(0,\sigma^{2})$.
\item $\gamma_{a,\sigma^2}$ is FQID, rather than FID. Moreover, it is absolutely continuous with respect to the Lebesgue measure.
\item $\gamma_{a,\sigma^2}=\gamma_{1,(\sigma/\sqrt{a})^2}^{\boxplus a}=\gamma_{a/\sigma^2,1}^{\boxplus \sigma^2}$ for $a\ge 1$ and $\sigma\ge 1$. 
\item Its free characteristic triplet is given by $(-\sigma^2, \frac{a\rmd u}{\pi u^2},0)$.
\item Its free characteristic pair is $(0, \frac{a\rmd u}{\pi (1+u^2)}-\sigma^2\delta_0(\rmd u))$.
\end{enumerate}
\end{thm}

In particular, $\gamma_{2\sqrt{2}\sigma,\sigma^2}$ (the case in which $a=2\sqrt{2}\sigma$) coincides with the measure $\rho(\mathbf{S}(0,\sigma^2))$ in \eqref{FQIDex}.

\begin{rem}\label{rem:Gaussianpart}
\begin{enumerate}[\rm(1)]
\item The FQID distribution $\gamma_{a,\sigma^2}$ has a free Gaussian part that takes negative values. However, if we consider the case of $a=0$, then $\varphi(z)=-\frac{\sigma^2}{z}$ is not the Voiculescu transform of any probability measure.

\item In general, for a classical characteristic triplet the Gaussian part must be nonnegative, see \cite[Lemma~2.7]{LPS}. Therefore the triplet $(-\sigma^2, \frac{a \rmd u}{\pi u^2}, 0)$ in Theorem \ref{C-S} is not a classical characteristic triplet.
\end{enumerate}
\end{rem}

\begin{ex}

Consider $a=1$ and $\sigma=1/2$. The Cauchy transform of $\gamma_{1,(1/2)^{2}}$ is given by
\begin{align*}
G_{\gamma_{1,(1/2)^2}}(z) = \frac{2}{z+i+\sqrt{(z+i)^{2}  + 1}}
= -2\left(z + i -\sqrt{z^{2} +2zi}\right), \qquad z\in\C^+.
\end{align*}
By the Stieltjes inversion formula, 
the probability density function $f(x)$ of $\gamma_{1,(1/2)^2}$ is
\begin{align*}
f(x) &= -\frac{1}{\pi}\lim_{y\downarrow 0}\Im(G_{\gamma_{1,(1/2)^2}}(x+iy))\\
&= \frac{2}{\pi} \left(1- \Im\left[\sqrt{x^{2} +2xi}\right]\right)\\
&=\frac{\sqrt{2}}{\pi}\left( \sqrt{2}-\sqrt{|x|\sqrt{x^2+4}-x^2}\right), \qquad x\in\R.
\end{align*}
Figure 2 shows the shape of the probability density function of $\gamma_{1,(1/2)^2}$. Clearly, it is not differentiable at $x = 0$.
\begin{figure}[ht]
 \begin{center}
  \includegraphics[width=70mm]{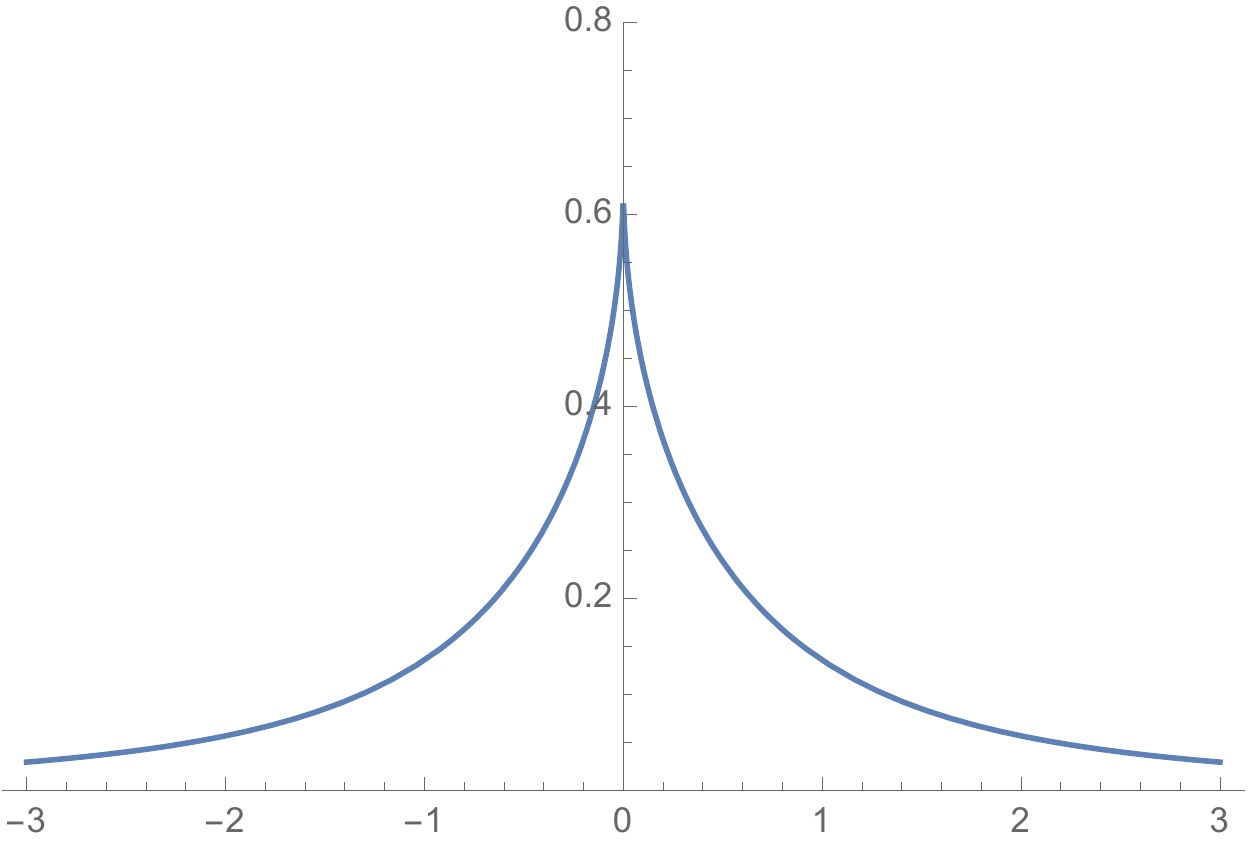}
 \end{center}
 \caption{Probability density function of $\gamma_{1,(1/2)^2}$}
 \label{fig:two}
\end{figure}

\end{ex}

\subsubsection{Case of the free Meixner distribution}

We define ${\bf FM}_{a,b}$ as the free Meixner distribution:
\begin{align*}
{\bf FM}_{a,b}(dx):= \frac{\sqrt{4(1+b)-(x-a)^2}}{2\pi (bx^2+ax+1)}\mathbf{1}_{[a-2\sqrt{1+b},a+2\sqrt{1+b}]}(x)dx + \text{ 0, 1, or 2 atoms},
\end{align*}
for $a\in\mathbb{R}$ and $b\ge -1$. 
If $b=-1$, then
\begin{align*}
{\bf FM}_{a,-1}=\frac{1}{2}\left(1+\frac{a}{\sqrt{1+a^2}}\right)\delta_{\frac{1}{2}(a-\sqrt{4+a^2})}+\frac{1}{2}\left(1-\frac{a}{\sqrt{1+a^2}}\right)\delta_{\frac{1}{2}(a+\sqrt{4+a^2})}
\end{align*}
is not FQID for all $a \in \R$ by Remark \ref{Bernoulli}. Moreover, it follows from \cite[Proposition 2.1]{BB} that, for $-1<b<0$, 
\begin{align*}
{\bf FM}_{a,b}=\mathbf{D}_{\sqrt{|b|}}\left( {\bf FM}_{\frac{a}{\sqrt{|b|}},-1}^{\boxplus t}\right), \qquad t=-\frac{1}{b}.
\end{align*}
Therefore, in view of Proposition 3.3(ii), ${\bf FM}_{a,b}$ is not FQID for $-1\le b<0$.
Note that ${\bf FM}_{a,b}\in {\rm ID}(\boxplus)$ if and only if $a\in \R$ and $b\ge 0$ (see \cite[Theorem 3.2]{SY}). Furthermore, the variance of ${\bf FM}_{0,b}$ is equal to 1 for all $b\ge 0$ (in particular, we have ${\bf FM}_{0,0}={\bf S}$). Using \cite[Theorem 1.2]{ATV}, for any $b\ge 0$, there exists an FQID distribution $\rho({\bf FM}_{0,b})$ (in \eqref{FQIDex}) such that 
\begin{align*}
{\bf FM}_{0,b}\boxplus \rho({\bf FM}_{0,b})=\mathbf{C}_{2\sqrt{2}}.
\end{align*}
We summarize the above result as follows.

\begin{thm}\label{ex:FMeixner}
For any $b\ge  0$, there exists some $\rho({\bf FM}_{0,b})$ such that ${\bf FM}_{0,b}\boxplus \rho({\bf FM}_{0,b})=\mathbf{C}_{2\sqrt{2}}$.  For $b\ge0$, the free Gaussian part of the measure $\rho({\bf FM}_{0,b})$ is equal to $-1$, and so the measure $\rho({\bf FM}_{0,b})$ is not FID. Furthermore, for $b>0$, the free quasi-L\'{e}vy measure $\nu_b$ of the probability measure $\rho({\bf FM}_{0,b})$ is given by
\begin{align*}
\nu_b(dx)=\begin{cases}
\frac{1}{\pi x^2}\left( 2\sqrt{2}- \frac{\sqrt{4b-x^2}}{2b } \right) dx, & x\in [-2\sqrt{b},2\sqrt{b}]\setminus \{0\}\\
\frac{2\sqrt{2}}{\pi x^2}dx, & x\in \mathbb{R}\setminus [-2\sqrt{b},2\sqrt{b}].
\end{cases}
\end{align*}
In particular, we can state the following:
\begin{enumerate}[\rm(1)]
\item If $b\ge 1/8$, then the density function $\frac{d\nu_b}{dx}(x)$ is always positive;

\item If $0<b<1/8$, then there exists some $x\in [-2\sqrt{b},2\sqrt{b}]\setminus\{0\}$ such that the density function $\frac{d\nu_b}{dx}(x)$ of the measure $\nu_b$ is negative. In this case, we have $\nu_b(\R)=-\infty$. For example, we show the density function of the free quasi-L\'{e}vy measure $\nu_{1/16}$ in Figure \ref{fig:three}; 

\item If $b=0$, then the free quasi-L\'{e}vy measure of $\rho({\bf FM}_{0,b})$ coincides with the free L\'{e}vy measure of the Cauchy distribution ${\bf C}_{2\sqrt{2}}$.
\end{enumerate}
\begin{figure}[htbp]
 \centering
  \includegraphics[width=100mm]{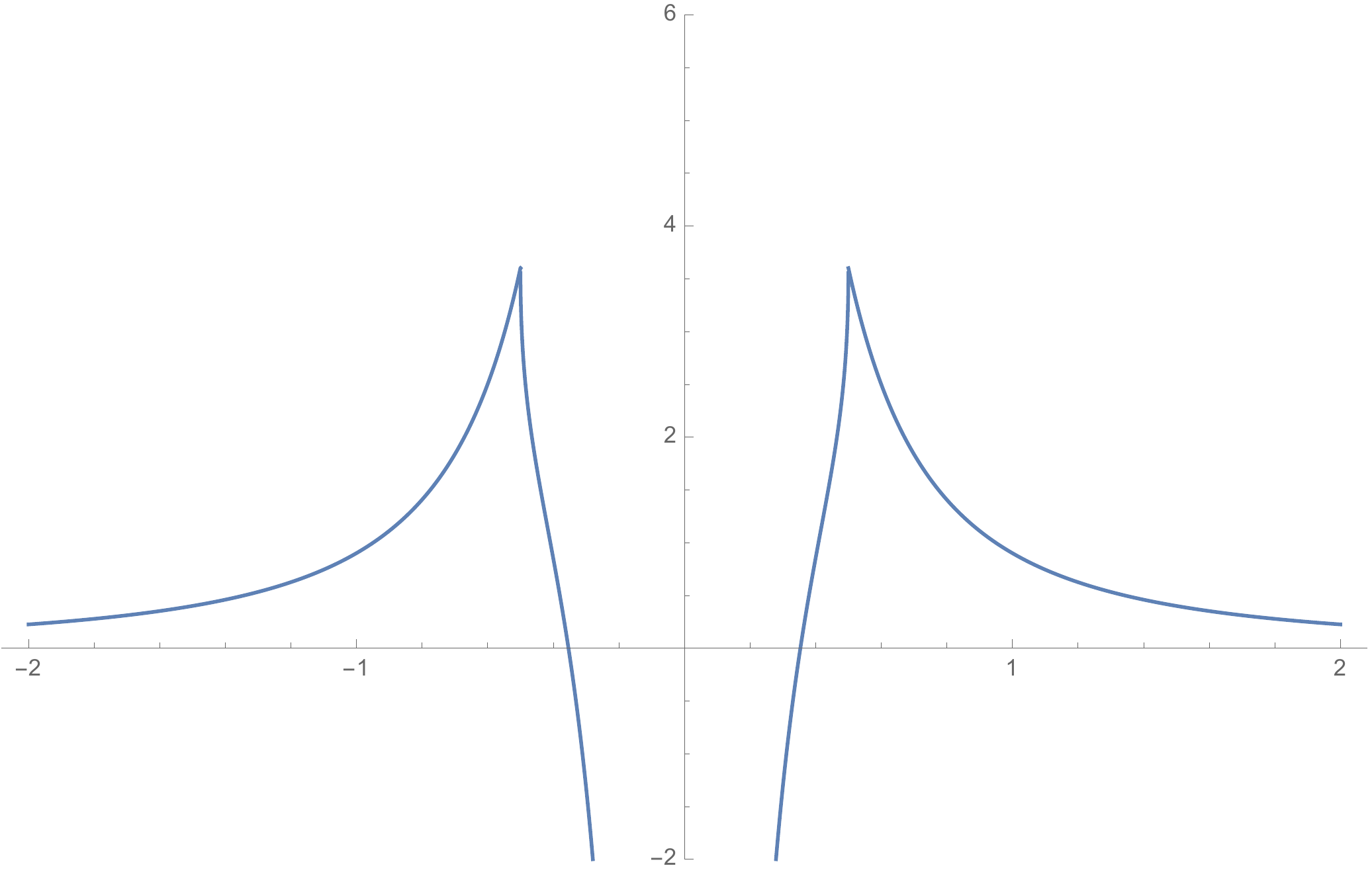}
 \caption{Density function of the free quasi-L\'{e}vy measure $\nu_{1/16}$}
 \label{fig:three}
\end{figure}
\end{thm}

\subsection{Free deconvolution of the semicircle law and failure of Cram\'{e}r's theorem in free probability}

For $t>0$, denote
\[
R_{t}^{\pm}(z):=z^2\pm tz^4\sum_{n=0}^{\infty}t^n z^{2n}=\frac{z^2-tz^4\pm tz^4}{1-tz^2}.
\]
Observe that
\[
R_{t}^{+}(z)
=\frac{z^2}{1-tz^2}
=\frac{1}{2\sqrt{t}}\cdot\frac{z\sqrt{t}}{1-z\sqrt{t}}+\frac{1}{2\sqrt{t}}\cdot\frac{-z\sqrt{t}}{1+z\sqrt{t}}.
\]
This means that $R_{t}^{+}(z)$ is the $R$-transform of a probability distribution $\mu_{+}(t)$,
which can be expressed as
\[
\mu_{+}(t)=\mathbf{MP}\left(\sqrt{t},\frac{1}{2\sqrt{t}}\right)\boxplus
\mathbf{MP}\left(-\sqrt{t},\frac{1}{2\sqrt{t}}\right).
\]
In particular, $\mu_{+}(t)$ is FID for all $t>0$.

For the function $R_{t}^{-}(z)$, we apply a result by Bercovici and Voiculescu \cite[Theorem 2]{BV},
which implies that if $t>0$ is sufficiently small, say $0<t<t_0$,
then $R_{t}^{-}(z)$ is the $R$-transform of some compactly supported
probability distribution $\mu_{-}(t)$ on $\mathbb{R}$. Then, for $0<t<t_0$, we have
$R_{\mu_{+}(t)}(z)+R_{\mu_{-}(t)}=2z^2$, which means that
$\mu_{+}(t)\boxplus\mu_{-}(t)=\mathbf{S}(0,2)$, and consequently,
\[
\mu_{-}(t)=\mathbf{S}(0,2)\boxminus\mu_{+}(t).
\]
In view of \cite[Remark 5]{BV}, the measure $\mu_{-}(t)$ cannot be FID,
which is also a consequence of the fact that the fourth free cumulant of $\mu_{-}(t)$ is $-t<0$.

This example illustrates the observation of Bercovici and Voiculescu \cite{BV}
that the free analog of Cram\'{e}r's decomposition theorem \cite[Theorem XV.8.1]{Fel}
(that the Gaussian distribution cannot be expressed as the classical convolution
of two non-Gaussian distributions) is not true.

\subsection{Free deconvolution with the Marchenko-Pastur law}
\subsubsection{Case of the semicircle law and the Marchenko-Pastur law}\label{subse:smp}

In view of  \cite[Proposition~5.1]{Mlo2021}, for every $u,x\in\mathbb{R}$, $u\ne0$,
there exists a probability distribution $\mu$ such that
\[
R_{\mu}(z)= x^3 u^2 z^2+(1-x)^3\frac{uz}{1-uz}.
\]
This means, that
\[
\left(u^2 x^3,(1-x)^3\delta_{u},0\right)
\]
is a free characteristic triplet for the measure $\mu$.
If $x<0$ then we have
\[
\mu=\mathbf{MP}\left(u,(1-x)^3\right)\boxminus\mathbf{S}\left(0,-x^3 u^2\right),
\]
while for $x>1$
\[
\mu=\mathbf{S}\left(0,x^3u^2\right)\boxminus\mathbf{MP}\left(u,(x-1)^3 u^2\right).
\]

If $x<0$, then the Gaussian part $u^2x^3$ is negative, and therefore $(u^2x^3, (1-x)^3\delta_u,0)$ is not a classical characteristic triplet (see \cite[Lemma 2.7]{LPS}). The following lemma, which was proved within \cite[Example~2.9]{LPS}, implies that
if $x>1$ then $\left(u^2 x^3,(1-x)^3\delta_{u},0\right)$
is not a classical characteristic triplet.

\begin{lem}\label{lem:lpsex29}
Assume that $\nu$ is a quasi-L\'{e}vy measure, with the positive and negative part $\nu^{+},\nu^{-}$. If
$\nu^{-}\ne0$ and if either $\nu^{+}=0$ or $\mathrm{supp}\,\nu^{+}$ is a one-point set, 
then $(a,\nu,\gamma)$ is not the characteristic triplet of a QID distribution
for any $a,\gamma\in\mathbb{R}$.
\end{lem}

\subsubsection{Case of two Marchenko-Pastur laws}\label{subse:twomp}

Assume that $u,v\in\mathbb{R}\setminus\{0\}$, $u<v$.
It was proven in \cite[Proposition~5.3]{Mlo2021} that for every $x\in\mathbb{R}$ there exists
a compactly supported probability measure $\mu$ such that
\[
R_{\mu}(z)=\frac{(u-x)^3}{u^2(u-v)}\cdot\frac{uz}{1-uz}+\frac{(v-x)^3}{v^2(v-u)}\cdot\frac{vz}{1-vz}.
\]
Putting
\begin{equation}\label{eq:twompab}
a(x):=\frac{(u-x)^3}{u^2(u-v)},\qquad
b(x):=\frac{(v-x)^3}{v^2(v-u)},
\end{equation}
we have $a(x)<0$, $b(x)>0$ for $x<u$, $a(x)$, $b(x)>0$ for $u<x<v$ and $a(x)>0$, $b(x)<0$ for $x>v$.
Therefore
\[
\mu=\mathbf{MP}\big(v,b(x)\big)\boxminus\mathbf{MP}\big(u,-a(x)\big)
\]
for $x<u$ and
\[
\mu=\mathbf{MP}\big(u,a(x)\big)\boxminus\mathbf{MP}\big(v,-b(x)\big)
\]
for $x>v$.
In another words, the free characteristic triplet for $\mu$ is $(0,\nu,0)$, where
\begin{equation}\label{eq:twompnu}
\nu:=a(x)\cdot\delta_{u}+b(x)\cdot\delta_{v}.
\end{equation}
Note that if $u+v\ne0$ then the sum of weights:
\[
a(x)+b(x)=\frac{u^2 v^2-3uv x^2+(u+v)x^3}{u^2 v^2}
\]
can be negative.

Again, by Lemma~\ref{lem:lpsex29}, if either $x<u$ or $x>u$, then $(0,\nu,0)$, with $\nu$ defined by \eqref{eq:twompab}, \eqref{eq:twompnu}, is not a classical characteristic triplet.






\subsubsection{Case of several Marchenko-Pastur laws}

Suppose that $u_1,\cdots,u_n\in\mathbb{R}\setminus\{0\}$ are distinct and $n\ge2$.
In view of \cite[Proposition~3.1 and Corollary~2.7]{LM}, there is a probability distribution
$\mu$ on $\R$ such that its $R$-transform is equal to
\begin{equation}\label{eq1}
R_{\mu}(z)=\frac{1-\prod_{i=1}^{n}(1-u_i z)}{\prod_{i=1}^{n}(1-u_i z)}.
\end{equation}
If we set
\begin{equation*}
t_k:=\frac{u_k^{n-1}}{\prod_{i\ne k}(u_k-u_i)},
\end{equation*}
then
\begin{equation}\label{eq3}
R_{\mu}(z)=\sum_{k=1}^{n} t_k\cdot \frac{u_k z}{1-u_k z}
\end{equation}
(see \cite[Lemma~3.5]{LM}).

In the case of $n=2$, if $u_1<u_2<0$, then $t_1>0>t_2$; if $u_1<0<u_2$, then $t_1,t_2>0$;
and if $0<u_1<u_2$, then $t_1<0<t_2$ (see \cite{LM} for details).
If $n\ge3$, then $t_i>0$ and $t_j<0$ for some $1\le i,j\le n$ (see \cite[Lemma~3.5]{LM}).

Now, assume that either $n=2$, $u_1 u_2>0$, or $n\ge3$. We observe that $\mu$ is FQID. Set
\[
K_{+}:=\{k:t_k>0\}:=\{k_1',\cdots,k_p'\},\qquad
K_{-}:=\{k:t_k<0\}:=\{k_1'',\cdots,k_q''\}
\]
and
\[
\mu_{+}:=\mu_{k_1'}\boxplus\cdots\boxplus\mu_{k_p'},\qquad
\mu_{-}:=\mu_{k_1''}\boxplus\cdots\boxplus\mu_{k_q''},
\]
where
\[
\mu_k:=\left\{\begin{array}{ll}
\mathbf{MP}(u_k,t_k)&\hbox{if $k\in K_+$,}\\
\mathbf{MP}(u_k,-t_k)&\hbox{if $k \in K_-$.}\\
\end{array}\right.
\]
Then, by (\ref{eq3}), we have the following deconvolution of $\mu$:
\[
\mu\boxplus\mu_{-}=\mu_{+};
\]
the characteristic triplet for $\mu$ is $(0,\nu,0)$, where
$\nu:=\sum_{k=1}^{n}t_k\cdot\delta_{u_k}$.
Note that $\nu(\mathbb{R})=1$.
Indeed, taking the limit $|z|\to\infty$ in (\ref{eq1})
and (\ref{eq3}), we find that $t_1+\cdots+t_n=1$.
For example, if $n=4$ and $u_k=k$, then
\[
t_1=-1/6,\quad
t_2=4,\quad
t_3=-27/2,\quad
t_4=32/3.
\]
One can check that if $n$ is even and the set $\{u_1,\cdots,u_n\}$ is symmetric,
i.e., $u_k=-u_{n-k+1}$, $1\le k\le n$, then $\nu$ is symmetric: $t_k=t_{n-k+1}$.

\subsection{Classical characteristic triplets which are not free triplets}

Take $\mu_{a}:=(1-a)\delta_{0}+a\delta_{1}$, with $0<a<1$, $a\ne1/2$.
Then $\mu_{a}$ is QID by \cite[Theorem~3.9]{LPS}.
More precisely, if $0<a<1/2$ then the characteristic triplet for $\mu_{a}$ is
$\tau_{a}:=(1,\nu_{a},0)$, where
\[
\nu_{a}=-\sum_{m=1}^{\infty}\frac{1}{m}\left(\frac{a}{a-1}\right)^{m}\delta_{m},
\]
while if $1/2<a<1$ then the characteristic triplet is $\tau_{a}:=(0,\nu_{a},0)$, with
\[
\nu_{a}=-\sum_{m=1}^{\infty}\frac{1}{m}\left(\frac{a-1}{a}\right)^{m}\delta_{-m}.
\]

\begin{prop}\label{prop:classtau}
If $0<a<1$, $a\ne1/2$, then $\tau_{a}$ is not a free characteristic triplet.
\end{prop}

\begin{proof}
The characteristic function of $\mu_{a}$ is given by
$$
\widehat{\mu_a}(z)=1-a+a\exp(iz).
$$
Since
\begin{align*}
\log\widehat{\mu_a}(-iz)&=\log\left(1-a+a\exp(z)\right)\\
&=az+\frac{1}{2}(a - a^2)z^2+\frac{1}{6}(a - 3 a^2 + 2 a^3)z^3+\frac{1}{24}(a - 7 a^2 + 12 a^3 - 6 a^4)z^4+\cdots,
\end{align*}
the classical cumulants of $\mu_{a}$ are
\[
r_1=a,\quad
r_2=a - a^2,\quad
r_3=a - 3 a^2 + 2 a^3,\quad
r_4=a - 7 a^2 + 12 a^3 - 6 a^4,\ldots.
\]

If $\tau_{a}$ was the free characteristic triplet for a probability distribution $\widetilde{\mu}_{a}$,
then $r_n$ were the free cumulants for $\widetilde{\mu}_{a}$.
This, by the moment-cumulant formula (see \cite[Proposition~11.4]{NiSpBook}), would imply, that the moments of $\widetilde{\mu}_{a}$ are
\[
s_0=1,\quad
s_1=s_2=s_3=a,\quad
s_4=a-a^2+2a^3-a^4,\ldots,
\]
which leads to a contradiction, because $\det(s_{i+j})_{i,j=0}^2=a^3(a-1)^3<0$.
\end{proof}

\section{An extension of the Bercovici-Pata bijection}
In the former section, we gave some examples, which are FQID, rather than FID. However their free characteristic triplet cannot be a classical characteristic triplet. That is, there does not exist probability measures whose classical characteristic triplets are same as their free characteristic triplets.

In this section, we show that we can extend the domain of the Bercovici-Pata bijection from an good example using P{\'o}yla's Theorem.

Define $\Phi$ as the set of such pairs $(c,\nu)$, that $c$ is a positive number
and $\nu$ is a symmetric L\'{e}vy measure, which satisfies
\begin{align}\label{integrable}
\int_{\mathbb{R}}\left(x^2\vee|x|\right)\nu(dx)<\infty.
\end{align}
Note that (\ref{integrable}) guarantees that the second moment $m_2(\nu)$ of $\nu$ is finite.
We define three subsets of $\Phi$, namely,
$\Phi^{+}$ is the set of these $(c,\nu)\in\Phi$ that the measure
\[
\frac{c}{\pi x^2}\,dx-\nu(dx)
\]
is nonnegative, and $\Phi^{*}$ (respectively, $\Phi^{\boxplus}$) will denote the set of such $(c,\nu)\in\Phi$
that
\[
\left(0,\frac{c}{\pi x^2}\,dx-\nu(dx),0\right)
\]
is a classical (respectively, free) characteristic triplet.
For $(c,\nu)\in\Phi^{*}$ (resp. $\Phi^{\boxplus}$) we denote by $\mu^{*}(c,\nu)$ (resp. $\mu^{\boxplus}(c,\nu)$)
the corresponding classical (resp. freely) quasi-infinitely divisible distribution.
We are going to show that the sets $\Phi^{*}\setminus\Phi^{+}$, $\Phi^{\boxplus}\setminus\Phi^{+}$
are nonempty, and, remarkably, that the set
$\Phi^{*}\cap\Phi^{\boxplus}\setminus\Phi^{+}$ is nonempty.
This will allow us to extend the Bercovici-Pata bijection, thus to answer
affirmatively a question raised by Bo\.zejko.

\begin{prop}\label{prop:convex}
If $(c,\nu)\in\Phi$ and if the function
\begin{equation}
\phi(z):=\exp\left(-c|z|+2\int_{0}^{+\infty}(1-\cos zx)\nu(dx)\right)
\end{equation}
is convex on $z\in[0,+\infty)$ then $(c,\nu)\in\Phi^{*}$ and $\phi$
is the characteristic function of the corresponding distribution $\mu^{*}(c,\nu)$. In particular, if
\begin{equation}\label{formula:phibis}
2\int_{0}^{\infty}x^2\cos zx\,\nu(dx)+\left(-c+2\int_{0}^{\infty}x\sin zx\,\nu(dx)\right)^2\ge0
\end{equation}
for all $z>0$ then $(c,\nu)\in\Phi^{*}$.
\end{prop}

\begin{proof}
Since
\[
\int_\R\left(e^{izx}-1-izx \mathbf{1}_{[-1,1]}(x)\right) \left(\frac{c}{\pi x^2} dx - \nu(dx)\right)
=-c|z| +2 \int_0^\infty (1-\cos zx)\nu(dx),
\]
the first statement is a consequence of P\'{o}lya's theorem, see \cite{Polya} or \cite[Section~XV.3]{Fel}.
Denoting the left hand side of (\ref{formula:phibis}) by $A(z)$ we have $\phi''(z)=A(z)\phi(z)$,
which implies the second statement.
\end{proof}

\begin{prop}\label{prop:secondmoment}
If $(c,\nu)\in\Phi$ and $c\ge\sqrt{8m_2(\nu)}$,
where $m_2(\nu)$ denotes the second moment of $\nu$,
then $(c,\nu)\in\Phi^{\boxplus}$.
\end{prop}

\begin{proof}
Since $\nu$ is a L\'{e}vy measure, there is a freely infinitely divisible probability measure $\mu_\nu$ such that
\begin{align*}
R_{\mu_\nu}(z) :&= \int_\R\left(\frac{1}{1-zx}-1-zx\mathbf{1}_{[-1,1]}(x) \right) \nu(dx)=2\int_0^\infty \left( \frac{1}{1-z^2x^2}-1\right) \nu(dx).
\end{align*}
Consider arbitrary numbers $\alpha,\beta>0$. As $z\rightarrow 0$ with $z\in \Delta_{\alpha,\beta}$, we get
\begin{align*}
R_{\mu_\nu}(z) &= m_2(\nu)z^2+2 \int_0^\infty \left( \frac{1}{1-z^2x^2}-1-z^2x^2\right)\nu(dx)\\
&=m_2(\nu)z^2 + o(z^2).
\end{align*}
Therefore, by \cite[Theorem 1.3]{BG06}, $\mu_\nu$ has a finite variance equal to $m_2(\nu)$
(note that $R_{\mu}(z)$ in \cite{BG06} denotes our $z^{-1} R_\mu (z)$).
Hence there is a probability measure $\rho(\mu_\nu)$ on $\R$ such that $\mu_\nu \boxplus \rho(\mu_\nu)=\mathbf{C}_{\sqrt{8m_2(\nu)}}$ by \cite[Theorem 1.2]{ATV}. This implies, that $(c,\nu)\in\Phi^{\boxplus}$.
\end{proof}

Now we restrict ourselves to a special case.

\begin{prop}\label{prop:cnu}
Assume that $\nu=p\left(\delta_{-\lambda}+\delta_{\lambda}\right)$, $p,\lambda,c>0$,
and put
\[
h(p):=\left\{\begin{array}{ll}
\sqrt{2p(4p+1)}&\hbox{if $0<p\le\frac{1}{4}$,}\\
2p+\sqrt{p}&\hbox{if $p>\frac{1}{4}$.}
\end{array}
\right.
\]
If $c\ge\lambda\cdot h(p)$ then $(c,\nu)\in\Phi^{*}$.
\end{prop}

\begin{proof}
Denoting again the left hand side of (\ref{formula:phibis}) by $A(z)$, and using inequality
$\cos\alpha\ge\frac{1}{2}\sin^2\alpha-1$, we get in our case
\begin{align*}
A(z)&=2p\lambda^2\cos\lambda z+\left(-c+2p\lambda\sin\lambda z\right)^2\\
&\ge(4p+1)p\lambda^2 \sin^2\lambda z-4p\lambda c\sin\lambda z-2p\lambda^2+c^2.
\end{align*}
To conclude it suffices to apply the following elementary observation:
if $a>0$, $|y|\le1$ and either $b^2\le 4ad$ (applied for $0<p\le1/4$) or $2a\le|b|\le a+d$
(applied for $p>1/4$), then $ay^2+by+d\ge0$.
\end{proof}

Putting $c:=\sqrt{8m_2(\nu)}=4\lambda\sqrt{p}$ and applying Proposition~\ref{prop:cnu} we obtain

\begin{cor}\label{cor:bpexample}
Assume that $0<4p\le9$, $\lambda>0$, $\nu:=p\left(\delta_{-\lambda}+\delta_{\lambda}\right)$
and $c:=\sqrt{8m_2(\nu)}=4\lambda\sqrt{p}$. Then $(c,\nu)\in\Phi^{*}\cap\Phi^{\boxplus}\setminus\Phi^{+}$.
Consequently,
\[
\left(0,\frac{c}{\pi x^2}\,dx-\nu(dx),0\right)
\]
is both a classical and a free characteristic triplet.
\end{cor}

This corollary allows us to extend the Bercovici-Pata bijection. Namely, define two families of distributions:
\begin{align*}
\mathcal{Q}(*)&:=\left\{\mu*\mu^{*}(c,\nu):\mu\in\mathrm{ID}(*),\,(c,\nu)\in\Phi^{*}\cap\Phi^{\boxplus}\right\},\\
\mathcal{Q}(\boxplus)&:=\left\{\mu\boxplus\mu^{\boxplus}(c,\nu):\mu\in\mathrm{ID}(\boxplus),\,(c,\nu)\in\Phi^{*}\cap\Phi^{\boxplus}\right\}.
\end{align*}
We have $\mathrm{ID}(*)\varsubsetneq\mathcal{Q}(*)$ and $\mathrm{ID}(\boxplus)\varsubsetneq\mathcal{Q}(\boxplus)$.
Note that the set $\Phi^{*}\cap\Phi^{\boxplus}$ is closed under addition, which implies that
$\mathcal{Q}(*)$ and $\mathcal{Q}(\boxplus)$ are closed under the classical and the free convolution respectively.
Now we can define an extension of the Bercovici-Pata bijection as $\widetilde{\Lambda}:\mathcal{Q}(*)\to\mathcal{Q}(\boxplus)$ by
\[
\widetilde{\Lambda}\left(\mu*\mu^{*}(c,\nu)\right):=\Lambda(\mu)\boxplus\mu^{\boxplus}(c,\nu),
\]
for $\mu\in\mathrm{ID}(*)$, $(c,\nu)\in\Phi^{*}\cap\Phi^{\boxplus}$.
This map is well defined, as the characteristic function and the $R$-transform characterize the corresponding
distribution uniquely.

\begin{rem}
Let $\mathcal{T}(*)$ and $\mathcal{T}(\boxplus)$ denote the set of classical and free characteristic triplets,
respectively, and let $\mathcal{T}$ denote the set of characteristic triplets corresponding to the (classically or freely) infinitely divisible
distributions. From Corollary~\ref{cor:bpexample} we see that $\mathcal{T}\varsubsetneq\mathcal{T}(*)\cap\mathcal{T}(\boxplus)$.
We have also seen that the set $\mathcal{T}(\boxplus)\setminus\mathcal{T}(*)$ is nonempty (Remark~\ref{rem:Gaussianpart}, subsections~\ref{subse:smp} and \ref{subse:twomp})
and the set  $\mathcal{T}(*)\setminus\mathcal{T}(\boxplus)$ is nonempty (Proposition~\ref{prop:classtau}).
\end{rem}


\end{document}